\documentclass[11pt,letterpaper,reqno]{amsart}

\usepackage[T1]{fontenc}
\usepackage{amsmath,amssymb,amsthm,mathtools}
\usepackage[margin=1.08in]{geometry}
\usepackage{microtype}
\usepackage{xcolor}
\colorlet{BLACK}{black}
\IfFileExists{libertinus.sty}{\usepackage{libertinus}}{}

\usepackage{enumitem}
\usepackage{hyperref}
\usepackage{aliascnt}
\usepackage[nameinlink,capitalize]{cleveref}

\hypersetup{
  colorlinks=true,
  linkcolor=black,
  citecolor=black,
  urlcolor=black,
  pdftitle={The Geometry of Real Anisotropic Bohnenblust--Hille Constants},
  pdfauthor={Daniel Nunez-Alarcon, Daniel M. Pellegrino, J. Santos, D. Serrano-Rodr\'iguez, and Eduardo V. Teixeira}
}
\setlist{itemsep=0.25em,topsep=0.45em}
\numberwithin{equation}{section}

\theoremstyle{plain}
\newtheorem{mainresultA}{Theorem}

\newtheorem{mainresultB}{Theorem}

\newtheorem{mainresultC}{\textcolor{black}{Theorem}}

\newtheorem{mainresultD}{Theorem}

\newtheorem{mainresultE}{Theorem}

\theoremstyle{plain}

\newaliascnt{proposition}{theorem}
\newtheorem{proposition}[proposition]{Proposition}
\aliascntresetthe{proposition}
\newaliascnt{lemma}{theorem}
\newtheorem{lemma}[lemma]{Lemma}
\aliascntresetthe{lemma}
\newaliascnt{corollary}{theorem}
\newtheorem{corollary}[corollary]{Corollary}
\aliascntresetthe{corollary}

\theoremstyle{definition}
\newaliascnt{definition}{theorem}

\aliascntresetthe{definition}

\theoremstyle{remark}
\newaliascnt{remark}{theorem}

\aliascntresetthe{remark}

\crefname{theorem}{theorem}{theorems}
\Crefname{theorem}{Theorem}{Theorems}
\crefname{proposition}{proposition}{propositions}
\Crefname{proposition}{Proposition}{Propositions}
\crefname{lemma}{lemma}{lemmas}
\Crefname{lemma}{Lemma}{Lemmas}
\crefname{corollary}{corollary}{corollaries}
\Crefname{corollary}{Corollary}{Corollaries}
\crefname{definition}{definition}{definitions}
\Crefname{definition}{Definition}{Definitions}
\crefname{remark}{remark}{remarks}
\Crefname{remark}{Remark}{Remarks}

\newcommand{\R}{\mathbb R}
\newcommand{\mult}{\mathrm{mult}}

\title{The Geometry of Real Anisotropic Bohnenblust--Hille Constants}

\author[D. N\'u\~nez-Alarc\'on]{Daniel N\'u\~nez-Alarc\'on}
\address{Department of Mathematics, Universidad Nacional de Colombia, Bogot\'a, Colombia}
\email{dnuneza@unal.edu.co}

\author[D. M. Pellegrino]{Daniel M. Pellegrino}
\address{Department of Mathematics, Universidade Federal da Para\'iba, Jo\~ao Pessoa, PB, Brazil}
\email{dmpellegrino@gmail.com}

\author[J. Santos]{\textcolor{black}{Joedson Silva dos Santos}}
\address{Department of Mathematics, Universidade Federal da Para\'iba, Jo\~ao Pessoa, PB, Brazil}
\email{joedson.santos@academico.ufpb.br}

\author[D. Serrano-Rodr\'iguez]{\textcolor{black}{Diana Marcela Serrano-Rodr\'iguez}}
\address{Department of Mathematics, Universidad Nacional de Colombia, Bogot\'a, Colombia}
\email{diserranor@unal.edu.co}

\author[E. V. Teixeira]{Eduardo V. Teixeira}
\address{Department of Mathematics, Oklahoma State University, Stillwater, OK, USA}
\email{eduardo.teixeira@okstate.edu}
\date{}


\allowdisplaybreaks
\begin{document}

\begin{abstract}
We determine the growth scale of the optimal constants in the real anisotropic
Bohnenblust--Hille inequality.  For an exponent vector $\mathbf q^{(m)}$, write
$C_{\mathbf q^{(m)}}^{(m)}$ for its optimal constant and $d_m$ for its diameter.
These constants are superpolynomial precisely when
$m d_m/\log m\to\infty$; throughout this regime, their logarithm has the
sharp scale $m d_m$.  If also $d_m\to0$, then
$\log C_{\mathbf q^{(m)}}^{(m)}/(m d_m)$ lies asymptotically in the interval
\[
 \left[\frac{\log2}{4},\frac{2-\log2-\gamma}{4}\right],
\]
where $\gamma$ is the Euler--Mascheroni constant; this interval has width less than
$10^{-2}$.  We also solve the extremal problems at
fixed diameter and fixed total deficit.  The normalized reciprocal-deficit
profiles order the \textcolor{black}{canonically arranged} optimal constants by majorization, yield exact formulas on
a full-dimensional region, and recover
the lower coefficient $(\log2)/4$ throughout a broad class of anisotropic
regimes.\end{abstract}

\maketitle

\section{Introduction and main results}\label{sec:introduction}

The Bohnenblust--Hille inequalities originated in the study of Dirichlet series. In their solution of Bohr's problem on absolute convergence, Bohnenblust and Hille~\cite{BohnenblustHille1931} obtained polynomial and multilinear inequalities that provide dimension-free estimates for the coefficients in terms of the uniform norm. The multilinear inequality, which is the setting considered here, asserts that the coefficients of an $m$-linear form on $\ell_\infty$ satisfy an $\ell_{2m/(m+1)}$ estimate with a constant independent of the dimension.

The quantitative behavior of these constants has become an important part of the theory. Besides its intrinsic interest, good control of the Bohnenblust--Hille constants has consequences in quantum information, notably in the study of multiplayer XOR games, and related Bohnenblust--Hille inequalities have more recently appeared in Boolean and noncommutative settings, with applications to learning low-degree functions and quantum observables \cite{AraujoPellegrino2019,ArunachalamDuttEscuderoPalazuelos2025,Montanaro2012,SloteVolbergZhang2024}. In the anisotropic multilinear theory the situation is richer: the single exponent is replaced by a vector of admissible exponents, and the optimal constant may vary substantially across this region. Our purpose is to understand this variation and, in particular, how the distribution of the exponents determines the size and asymptotic growth of the optimal constants.

Let $m\ge2$, let
$\mathbb K\in\{\R,\mathbb C\}$, and write $c_0$ for the Banach space of scalar
sequences converging to zero, with its usual supremum norm.  If
$(e_j)_{j\ge1}$ is the canonical basis and $T:c_0^m\to\mathbb K$ is continuous,
set
\[
 \|T\|=\sup\bigl\{|T(x^{(1)},\ldots,x^{(m)})|:\|x^{(i)}\|_\infty\le1,\ 1\le i\le m\bigr\}.
\]
The multilinear Bohnenblust--Hille inequality asserts that
\[
 \left(\sum_{j_1,\ldots,j_m=1}^{\infty}
 |T(e_{j_1},\ldots,e_{j_m})|^{\frac{2m}{m+1}}\right)^{\frac{m+1}{2m}}
 \le B_{\mathbb K,m}^{\mult}\|T\|
\]
for every such $T$, where $B_{\mathbb K,m}^{\mult}$ denotes the optimal constant.  The isotropic exponent $2m/(m+1)$ is only one point of a larger
anisotropic family, characterized in \cite{ABPS2014}.  From now on the scalar
field is $\R$, and we write
\begin{equation}\label{eq:BH-condition}
 \mathcal B_m
 :=\left\{\mathbf q=(q_1,\ldots,q_m)\in[1,2]^m:
 \frac{1}{q_1}+\cdots+\frac{1}{q_m}\le\frac{m+1}{2}\right\}
\end{equation}
for the admissible exponent vectors.  We refer to the elements of $\mathcal B_m$ as Bohnenblust--Hille exponent vectors, and call $\mathbf q\in\mathcal B_m$ \emph{critical} when equality holds in \eqref{eq:BH-condition}.  For a finitely
supported scalar array $a=(a_{j_1,\ldots,j_m})$, its mixed norm is
\[
 \|a\|_{\ell_{q_1}(\ell_{q_2}(\cdots\ell_{q_m}))}
 :=\left(\sum_{j_1}
 \left(\sum_{j_2}\cdots
 \left(\sum_{j_m}|a_{j_1,\ldots,j_m}|^{q_m}\right)^{q_{m-1}/q_m}
 \cdots\right)^{q_1/q_2}\right)^{1/q_1}.
\]
For an arbitrary array, the same notation denotes the supremum of these
norms over all finite rectangular truncations.
By \cite[Theorem~1.1]{ABPS2014}, membership in $\mathcal B_m$ is equivalent to
the existence of a constant $C\ge1$ such that
\begin{equation}\label{eq:anisotropic-BH}
 \left\|\bigl(T(e_{j_1},\ldots,e_{j_m})\bigr)\right\|_{\ell_{q_1}(\ell_{q_2}(\cdots\ell_{q_m}))}\le C\|T\|
\end{equation}
for every continuous $m$-linear form $T:c_0^m\to\R$.  We denote the least
constant in \eqref{eq:anisotropic-BH} by $C_{\mathbf q}^{(m)}$.

The problem is not merely to decide which exponent vectors are admissible,
but to understand how the best constants vary across $\mathcal B_m$ as the
degree grows.  Two familiar points already reveal the breadth of the answer.
At the diagonal exponent $q_i=2m/(m+1)$, the constants have a sublinear
upper bound \cite[Corollary~3.3]{BayartPellegrinoSeoane2014}; at the
mixed endpoint,
\[
 C_{(1,2,\ldots,2)}^{(m)}=2^{(m-1)/2}
\]
\cite[Theorem~2.1]{PellegrinoJNT2016}.  Thus the same admissible region contains
both mild and exponential behavior.

For $\mathbf q=(q_1,\ldots,q_m)\in\mathcal B_m$, define its diameter by
\[
 d(\mathbf q):=\max_{1\le i\le m}q_i-\min_{1\le i\le m}q_i.
\]
For a sequence $(\mathbf q^{(m)})_{m\ge2}$, with $\mathbf q^{(m)}=(q_1^{(m)},\ldots,q_m^{(m)})\in\mathcal B_m$, we write
\[
 d_m:=d(\mathbf q^{(m)})
 =\max_{1\le i\le m}q_i^{(m)}-\min_{1\le i\le m}q_i^{(m)}.
\]
The diameter was introduced in this context in
\cite{PellegrinoTeixeira2018}.  It was subsequently proved that $d_m\to0$
characterizes subexponential growth, and the superpolynomial problem was
settled for $k$-mixed vectors \cite[Theorem~1.5 and
Section~5]{CostaNunezPellegrinoRaposo2026}.  Here subexponential means
$\log C_{\mathbf q^{(m)}}^{(m)}/m\to0$, or equivalently
$C_{\mathbf q^{(m)}}^{(m)}/a^m\to0$ for every $a>1$.

A positive sequence $(c_m)$ will be called \emph{superpolynomial} if
\[
 \frac{c_m}{m^A}\longrightarrow\infty
 \qquad\text{for every }A>0,
\]
equivalently, if $\log c_m/\log m\to\infty$.

Our first result identifies the precise transition.  The diameter alone gives
both the sharp superpolynomial threshold and the correct logarithmic scale
above it.  For positive sequences $(u_m)$ and $(v_m)$, we write
$u_m\asymp v_m$ if there are constants $0<c\le C<\infty$ such that
\[
c v_m\le u_m\le C v_m
\]
for all sufficiently large $m$.  Let $\gamma$ be the Euler--Mascheroni
constant, and set
\begin{equation}\label{eq:intro-vartheta-definition}
\vartheta:=\frac{2-\log2-\gamma}{2}.
\end{equation}

\begin{mainresultA}[Growth determined by the diameter]\label{thm:main-A}
Let $(\mathbf q^{(m)})_{m\ge2}$ be a sequence of Bohnenblust--Hille exponents and put $d_m=d(\mathbf q^{(m)})$.
\begin{enumerate}[label=\textup{(\roman*)},leftmargin=2.35em]
\item The exact threshold for superpolynomial growth is
\begin{equation}\label{eq:intro-superpoly}
 \frac{\bigl(\log C_{\mathbf q^{(m)}}^{(m)}\bigr)}{\log m}\longrightarrow\infty
 \quad\Longleftrightarrow\quad
 \frac{m d_m}{\log m}\longrightarrow\infty.
\end{equation}

\item If $m d_m/\log m\to\infty$, then
\begin{equation}\label{eq:intro-global-growth-law}
 \frac{\log2}{4}
 \le \liminf_{m\to\infty}\frac{\bigl(\log C_{\mathbf q^{(m)}}^{(m)}\bigr)}{m d_m}
 \le \limsup_{m\to\infty}\frac{\bigl(\log C_{\mathbf q^{(m)}}^{(m)}\bigr)}{m d_m}
 \le \frac{\log2}{2},
\end{equation}
and hence $\log C_{\mathbf q^{(m)}}^{(m)}\asymp m d_m$ throughout the superpolynomial region.

\item If, in addition, $d_m\to0$, then
\begin{equation}\label{eq:intro-corridor}
 \frac{\log2}{4}
 \le \liminf_{m\to\infty}\frac{\bigl(\log C_{\mathbf q^{(m)}}^{(m)}\bigr)}{m d_m}
 \le \limsup_{m\to\infty}\frac{\bigl(\log C_{\mathbf q^{(m)}}^{(m)}\bigr)}{m d_m}
 \le \frac{\vartheta}{2}.
\end{equation}
\end{enumerate}
\end{mainresultA}

The theorem places the critical diameter scale at $(\log m)/m$.  Below this
scale one cannot have superpolynomial growth; above it, the logarithm of the
optimal constant is comparable to $m d_m$.  The global upper coefficient
$(\log2)/2$ is sharp, as the endpoint $(1,2,\ldots,2)$ shows.

\begin{mainresultB}\label{thm:main-B-fixed}
For every $m\ge2$ and $0\le d\le1$,
\begin{equation}\label{eq:intro-fixed-diameter-envelope}
 \min_{\substack{\mathbf q\in\mathcal B_m\\ d(\mathbf q)=d}}
 C_{\mathbf q}^{(m)}
 =2^{\frac{(m-1)d}{2(2-d)}}.
\end{equation}
The minimizers are precisely the permutations of $(2-d,2,\ldots,2)$. Moreover, the lower coefficient $(\log2)/4$ in Theorem~\ref{thm:main-A}(iii) is optimal: if $d_m>0$, $d_m\to0$, and $m d_m/\log m\to\infty$, then for $\mathbf q^{(m)}=(2-d_m,2,\ldots,2)$,
\[
 \frac{\log C_{\mathbf q^{(m)}}^{(m)}}{m d_m}\longrightarrow\frac{\log2}{4}.
\]
\end{mainresultB}

Theorem~\ref{thm:main-B-fixed} identifies the smallest optimal constant among
exponent vectors with prescribed diameter.  The diameter does not determine
the optimal constant: for $0<d\le2/3$, consider
\[
 \mathbf q_m^{(1)}(d)=(2-d,2,\ldots,2),\qquad
 \mathbf q_m^{(2)}(d)=(2-d,2-d,2,\ldots,2).
\]
Both vectors are admissible and have diameter $d$.  Theorem~\ref{thm:main-B-fixed}
gives the exact constant for the first family, while Theorem~\ref{thm:main-D}
below gives the exact constant for the second.  Thus, with $t=d/(2-d)$,
\[
 C_{\mathbf q_m^{(1)}(d)}^{(m)}=2^{t(m-1)/2},
 \qquad
 C_{\mathbf q_m^{(2)}(d)}^{(m)}=2^{tm/2}.
\]
The discrepancy comes from information that the diameter discards: how the
reciprocal deficit is distributed among the coordinates.  

Define the reciprocal-deficit coordinates
\[
 x:\mathcal B_m\longrightarrow[0,1]^m,
 \qquad
 x(\mathbf q):=(x_1(\mathbf q),\ldots,x_m(\mathbf q)),
\]
by
\begin{equation}\label{eq:intro-deficit}
 x_i(\mathbf q):=2\left(\frac{1}{q_i}-\frac{1}{2}\right),
 \qquad 1\le i\le m,
\end{equation}
and define the total reciprocal deficit
\[
 \sigma:\mathcal B_m\longrightarrow[0,1],
 \qquad
 \sigma(\mathbf q):=\sum_{i=1}^m x_i(\mathbf q).
\]
The condition $\sigma(\mathbf q)\le1$ is exactly the admissibility condition
in \eqref{eq:BH-condition}.  Thus $\sigma$ measures the total distance from
the Hilbertian point in reciprocal coordinates.  Let
\[
 \Delta_{m-1}:=\left\{z\in[0,\infty)^m:\sum_{i=1}^m z_i=1\right\}.
\]
For $\mathbf q\ne(2,\ldots,2)$, normalize this deficit by setting
\begin{equation}\label{eq:intro-normalized-profile}
 y:\mathcal B_m\setminus\{(2,\ldots,2)\}\longrightarrow\Delta_{m-1},
 \qquad
 y(\mathbf q):=\frac{x(\mathbf q)}{\sigma(\mathbf q)}.
\end{equation}
We write $[m]:=\{1,\ldots,m\}$ and, for $S\subset[m]$, let $\mathbf1_S\in\{0,1\}^m$ denote its indicator vector.
For $z\in\Delta_{m-1}$, write $z^\downarrow=(z_1^\downarrow,\ldots,z_m^\downarrow)$ for the nonincreasing rearrangement of its coordinates. For $1\le r\le m$, define the ordered partial-sum map
\[
 X_r:\Delta_{m-1}\longrightarrow[0,1],
 \qquad
 X_r(z):=\sum_{i=1}^r z_i^\downarrow.
\]
For $\mathbf q\ne(2,\ldots,2)$, set
\begin{equation}\label{eq:intro-ordered-sums}
 Y_r(\mathbf q):=X_r(y(\mathbf q))
 =\sum_{i=1}^r y_i^\downarrow(\mathbf q),
 \qquad 1\le r\le m.
\end{equation}
Thus $Y_1(\mathbf q)$ is the largest share of the total deficit carried by a
single coordinate.  We use the standard majorization order on
$\Delta_{m-1}$: for $z,w\in\Delta_{m-1}$,
\[
 z\prec w
 \quad\Longleftrightarrow\quad
 X_r(z)\le X_r(w)\quad(1\le r<m).
\]
The equality at $r=m$ is automatic.

Fix $0<s\le1$.  To compare profiles on the slice $\sigma=s$, define
\[
 \mathbf Q_m:(0,1]\times\Delta_{m-1}\longrightarrow\mathcal B_m
\]
by
\begin{equation}\label{eq:intro-profile-reconstruction}
 \mathbf Q_m(s,z):=(Q_1(s,z),\ldots,Q_m(s,z)),
 \qquad
 \frac{1}{Q_i(s,z)}=\frac{1}{2}+\frac{s z_i}{2},
 \quad 1\le i\le m.
\end{equation}
Define
\[
 \mathfrak C_{m,s}:\Delta_{m-1}\longrightarrow[1,\infty)
\]
by
\begin{equation}\label{eq:intro-canonical-profile-constant}
 \mathfrak C_{m,s}(z)
 \textcolor{black}{:=\max_{\pi\in\mathfrak S_m}
 C_{\mathbf Q_m(s,\pi z)}^{(m)}.}
\end{equation}
\textcolor{black}{The maximum accounts for the order remembered by the iterated mixed norm.}

\begin{mainresultC}[Profile majorization]
\label{thm:main-C}
For every $m\ge2$ and $0<s\le1$, the function
$z\mapsto\log\mathfrak C_{m,s}(z)$ is symmetric, convex, and Schur-convex on
$\Delta_{m-1}$.  In particular,
\begin{equation}\label{eq:intro-canonical-order}
 \mathfrak C_{m,s}(z)
 =C_{\mathbf Q_m(s,z^\downarrow)}^{(m)}.
\end{equation}
Consequently,
\begin{equation}\label{eq:intro-actual-majorization}
 z\prec w
 \quad\Longrightarrow\quad
 \mathfrak C_{m,s}(z)\le\mathfrak C_{m,s}(w).
\end{equation}
In particular,
\begin{equation}\label{eq:intro-profile-extremes}
 C_{(2m/(m+s),\ldots,2m/(m+s))}^{(m)}
 \le \mathfrak C_{m,s}(z)
 \le 2^{(m-1)s/2}.
\end{equation}
\end{mainresultC}

\begin{mainresultD}[Exact profile formulas]\label{thm:main-D}
Let $m\ge2$.
\begin{enumerate}[label=\textup{(\roman*)},leftmargin=2.35em]
\item For every $0\le s\le1$,
\begin{equation}\label{eq:intro-total-deficit-envelope}
 \max_{\substack{\mathbf q\in\mathcal B_m\\ \sigma(\mathbf q)=s}}
 C_{\mathbf q}^{(m)}
 =2^{(m-1)s/2}.
\end{equation}
The maximum is attained by every permutation of
\[
 \mathbf q_s=\left(\frac{2}{1+s},2,\ldots,2\right).
\]

\item For every Bohnenblust--Hille exponent $\mathbf q$ with $\sigma(\mathbf q)>0$ and $Y_1(\mathbf q)\ge1/2$,
\begin{equation}\label{eq:intro-exact-dominant-profile}
 C_{\mathbf q}^{(m)}
 =2^{\frac{\sigma(\mathbf q)}{2}\left(1+(m-2)Y_1(\mathbf q)\right)}.
\end{equation}
\end{enumerate}
\end{mainresultD}

For $m=3$, the region $Y_1\ge1/2$ occupies three quarters of the profile simplex $\Delta_2$. Thus Theorem~\ref{thm:main-D}(ii) determines the optimal anisotropic Bohnenblust--Hille constant on $75\%$ of the profile region.

The second assertion depends on the full profile: on the region
$Y_1\ge1/2$, all remaining coordinates may be averaged by majorization
without changing the sharp constant.  This region has nonempty relative
interior in $\Delta_{m-1}$, so the formula is not confined to a lower-dimensional
boundary.  The first assertion points in the opposite direction: at fixed
total deficit, the most concentrated profile is extremal.

Let $(r_j)_{j\ge1}$ denote the Rademacher system on $[0,1]$. For $1\le p\le2$, let $A_p$ denote the optimal lower real Khinchine constant, that is, the largest constant such that, for every integer $N\ge1$ and every real scalar sequence $(a_j)_{j=1}^N$,
\begin{equation}\label{eq:intro-Khinchine-definition}
 A_p\left(\sum_{j=1}^N |a_j|^2\right)^{1/2}
 \le
 \left(\int_0^1\left|\sum_{j=1}^N a_j r_j(t)\right|^p\,dt\right)^{1/p}.
\end{equation}

For every integer $k\ge1$, define
\begin{equation}\label{eq:intro-pk-definition}
 p_k:=\frac{2k}{k+1}.
\end{equation}
If
\begin{equation}\label{eq:intro-flat-endpoint-constant}
 C_{m,k}:=C_{(\underbrace{p_k,\ldots,p_k}_{k},
                    \underbrace{2,\ldots,2}_{m-k})}^{(m)},
 \qquad 1\le k\le m,
\end{equation}
then \textcolor{black}{Theorem}~\ref{thm:main-C}, applied to the profiles
$k^{-1}\mathbf1_{\{1,\ldots,k\}}$, gives the monotone chain
\begin{equation}\label{eq:intro-flat-monotonicity}
 B_{\R,m}^{\mult}=C_{m,m}\le C_{m,m-1}\le\cdots\le C_{m,1}
 =2^{(m-1)/2}.
\end{equation}
We also denote by $p_{\mathrm H}$ Haagerup's breakpoint,
\begin{equation}\label{eq:intro-pH-definition}
 p_{\mathrm H}=1.847416\ldots,
 \qquad
 \Gamma\!\left(\frac{p_{\mathrm H}+1}{2}\right)=\frac{\sqrt\pi}{2}.
\end{equation}

The remaining gap in Theorem~\ref{thm:main-A}(iii) is governed by the sharp
Khinchine constants.  Since $p_{12}=24/13<p_{\mathrm H}<13/7=p_{13}$, the lower and upper extension factors agree through level $12$.  Beyond that point
their difference is small and can be measured exactly.  Decomposing the
profile over nested level sets gives

\begin{mainresultE}[Sharp asymptotics from the deficit tail]\label{thm:main-E}
Let $(\mathbf q^{(m)})_{m\ge2}$ be Bohnenblust--Hille exponents and put
$d_m=d(\mathbf q^{(m)})$. Assume
\[
 d_m\to0,
 \qquad
 \frac{m d_m}{\log m}\to\infty.
\]
Then
\begin{equation}\label{eq:intro-main-defect-tail-bound}
 \limsup_{m\to\infty}
 \frac{\bigl(\log C_{\mathbf q^{(m)}}^{(m)}\bigr)}{m d_m}
 \le
 \frac{\log2}{4}
 +\left(\frac{\vartheta}{2}-\frac{\log2}{4}\right)
 \limsup_{m\to\infty}
 \frac{y_{13}^\downarrow(\mathbf q^{(m)})}{Y_1(\mathbf q^{(m)})}.
\end{equation}
\end{mainresultE}

\textcolor{black}{In the asymptotic statements involving $y_{13}^\downarrow$, the finitely many indices $m<13$ are simply omitted.}

The index $13$ is the first level beyond the exact Haagerup range; see Section~\ref{sec:proof-main-E}.  The ratio $y_{13}^\downarrow/Y_1$ measures whether a thirteenth profile mass
remains comparable to the largest one.  The lower coefficient is therefore
sharp under either of the following conditions.

\begin{corollary}\label{cor:intro-mainD-collapse}
Under the assumptions of Theorem~\ref{thm:main-E}, one has
\[
 \frac{\bigl(\log C_{\mathbf q^{(m)}}^{(m)}\bigr)}{m d_m}
 \longrightarrow\frac{\log2}{4}
\]
whenever either of the following conditions holds:
\begin{enumerate}[label=\textup{(\alph*)},leftmargin=2.35em]
\item $y_1^\downarrow(\mathbf q^{(m)})\ge1/12$ for all sufficiently large $m$;
\item $y_{13}^\downarrow(\mathbf q^{(m)})/y_1^\downarrow(\mathbf q^{(m)})\to0$.
\end{enumerate}
\end{corollary}

\section{The diameter scale}\label{sec:known-estimates}

For a finitely supported array $a$ and $\mathbf q\in[1,2]^m$, we abbreviate
the mixed norm from Section~\ref{sec:introduction} by $\|a\|_{\mathbf q}$.
Its coordinatewise monotonicity gives
\[
 r_i\le q_i\ (1\le i\le m)
 \quad\Longrightarrow\quad
 \|a\|_{\mathbf q}\le \|a\|_{\mathbf r}.
\]
If $N_1,\ldots,N_m\ge1$ and
$T:\ell_\infty^{N_1}\times\cdots\times\ell_\infty^{N_m}\to\R$, let
$\widehat T=(T(e_{j_1},\ldots,e_{j_m}))$ be its coefficient array.  Then
\begin{equation}\label{eq:finite-dimensional-sup}
 C_{\mathbf q}^{(m)}=
 \sup_{N_1,\ldots,N_m\ge1}\ \sup_{0\ne T}
 \frac{\|\widehat T\|_{\mathbf q}}{\|T\|}.
\end{equation}

We also retain the lower Khinchine constants from
\eqref{eq:intro-Khinchine-definition}.  Haagerup's formula
\cite[Theorem~A]{Haagerup1981} reads
\begin{equation}\label{eq:Haagerup-formula}
 A_p=
 \begin{cases}
 2^{\frac{1}{2}-\frac{1}p},&1\le p\le p_{\mathrm H},\\[2mm]
 \displaystyle\sqrt2\left(\frac{\Gamma((p+1)/2)}{\sqrt\pi}\right)^{1/p},
 &p_{\mathrm H}\le p\le2,
 \end{cases}
\end{equation}
\subsection{\texorpdfstring{{Auxiliary estimates}}{Auxiliary estimates}}\label{subsec:Khinchine-estimates}

\begin{lemma}\label{lem:elementary-Ap-bound}
For every $1\le p\le2$,
\begin{equation}\label{eq:elementary-Ap-bound}
 A_p^{-1}\le \exp\left(\frac{1}p-\frac{1}{2}\right).
\end{equation}
In particular, for $p_k$ as in \eqref{eq:intro-pk-definition},
\begin{equation}\label{eq:Ap-pk-bound}
 A_{p_k}^{-1}\le e^{1/(2k)}.
\end{equation}
Moreover, for every $k\ge2$,
\begin{equation}\label{eq:Ap-pk-dyadic-bound}
 A_{p_k}^{-1}\le 2^{1/(k+2)}.
\end{equation}
\end{lemma}

\begin{proof}
If $1\le p\le p_{\mathrm H}$, the first branch of \eqref{eq:Haagerup-formula}
gives
\[
 \log A_p=\left(\frac{1}{2}-\frac{1}p\right)\log2
 \ge \frac{1}{2}-\frac{1}p,
\]
because $\frac{1}{2}-\frac{1}p\le0$ and $\log2<1$.

Suppose now that $p_{\mathrm H}\le p\le2$ and set $\xi=(p+1)/2$.  The second branch
gives
\begin{equation}\label{eq:elementary-Ap-log}
 \log A_p=\frac{1}{2}\log2+\frac{1}p\left(\log\Gamma(\xi)-\frac{1}{2}\log\pi\right).
\end{equation}
Convexity of $\log\Gamma$, applied at $3/2$, together with
$\Gamma(3/2)=\sqrt\pi/2$ and
$\psi(3/2)=2-\gamma-2\log2$, yields
\begin{align*}
 \log\Gamma(\xi)
 &\ge \log\Gamma(3/2)+\psi(3/2)(\xi-3/2)\\
 &=\frac{1}{2}\log\pi-\log2-\frac{2-p}{2}(2-\gamma-2\log2).
\end{align*}
Substituting this into \eqref{eq:elementary-Ap-log} and collecting terms gives
\[
 \log A_p-\left(\frac{1}{2}-\frac{1}p\right)
 \ge \frac{2-p}{2p}(\gamma+\log2-1).
\]
The last quantity is nonnegative.  Thus
$\log A_p\ge\frac{1}{2}-\frac{1}p$ on the second branch as well, proving
\eqref{eq:elementary-Ap-bound}.  Since
$1/p_k-1/2=1/(2k)$, \eqref{eq:Ap-pk-bound} follows.  If
$2\le k\le12$, then $p_k<p_{\mathrm H}$ and
$A_{p_k}^{-1}=2^{1/(2k)}\le2^{1/(k+2)}$.  If $k\ge13$, then
\eqref{eq:Ap-pk-bound} gives \eqref{eq:Ap-pk-dyadic-bound}, because
$k+2\le2k\log2$ for $k\ge6$.
\end{proof}

\begin{lemma}\label{lem:Khinchine-asymptotics}
{For $p_k$ as in \eqref{eq:intro-pk-definition}, with $\vartheta=(2-\log2-\gamma)/2$,}
\begin{equation}\label{eq:Khinchine-asymptotics}
 \lim_{k\to\infty}(-k\log A_{p_k})=\vartheta.
\end{equation}
\end{lemma}

\begin{proof}
For large $k$, the second branch of \eqref{eq:Haagerup-formula} applies and
$(p_k+1)/2=3/2-1/(k+1)$.  Taylor expansion at $3/2$ gives
\[
 \log\Gamma\!\left(\frac{p_k+1}{2}\right)
 =\log\Gamma(3/2)-\frac{\psi(3/2)}{k+1}+O(k^{-2}).
\]
Substituting the values of $\Gamma(3/2)$ and $\psi(3/2)$ into Haagerup's
formula, we find
\[
 \log A_{p_k}
 =-\frac{2-\log2-\gamma}{2(k+1)}+O(k^{-2})
 =-\frac{\vartheta}{k+1}+O(k^{-2}).
\]
Hence $-k\log A_{p_k}=\vartheta k/(k+1)+O(k^{-1})$, and the conclusion
follows.
\end{proof}

\medskip
The mixed-norm facts below are standard; see \cite[Proposition~3.1]{BayartPellegrinoSeoane2014} and
\cite[Proposition~2.1, Corollary~2.3 and Proposition~3.1]{ABPS2014}.

\begin{proposition}\label{prop:known-mixed-norm-tools}
Let $\mathbf q=(q_1,\ldots,q_m)$ be a Bohnenblust--Hille exponent.

\begin{enumerate}
\item If $p=\min_{1\le i\le m}q_i$ and $\mathbf q^+=(q_1,\ldots,q_m,2)$, then
\[
 C_{\mathbf q^+}^{(m+1)}\le A_p^{-1}C_{\mathbf q}^{(m)}.
\]

\item Suppose that two adjacent coordinates of $\mathbf q$ are $(q,p)$ with $1\le p\le q\le2$, and let $\mathbf q'$ be obtained by replacing them by $(p,q)$. Then
\[
 \textcolor{black}{C_{\mathbf q}^{(m)}\le C_{\mathbf q'}^{(m)}.}
\]
Consequently, if $\mathbf t$ is any permutation of the vector with $k$ copies of $p$ followed by $m-k$ copies of $2$, then
\[
 \textcolor{black}{C_{\mathbf t}^{(m)}
 \le C_{(\underbrace{p,\ldots,p}_{k},\underbrace{2,\ldots,2}_{m-k})}^{(m)}.}
\]

\item Let $2\le k\le m$. If $p_k\le q_i\le2$ for every $1\le i\le m$, then
\begin{equation}\label{eq:Khinchine-interpolation}
 C_{\mathbf q}^{(m)}
 \le A_{p_k}^{-(m-k)}B_{\R,k}^{\mult}.
\end{equation}
\end{enumerate}
\end{proposition}

\begin{proof}
Assertion (1) is the standard one-coordinate Khinchine extension: apply the lower Khinchine inequality in the new coefficient index and then the $m$-linear estimate for $\mathbf q$; see the references preceding the proposition.

\textcolor{black}{For (2), let $a$ be the coefficient array of $T$ and let
$a^{(i\,i+1)}$ be obtained by interchanging the two adjacent indices.
Minkowski's inequality gives}
\[
 \textcolor{black}{\|a\|_{(\ldots,q,p,\ldots)}
 \le
 \|a^{(i\,i+1)}\|_{(\ldots,p,q,\ldots)},
 \qquad p\le q.}
\]
\textcolor{black}{The array on the right belongs to the form obtained by
interchanging the same two variables, whose norm equals $\|T\|$. Taking
suprema proves the first assertion; repeated adjacent transpositions then
move every copy of $p$ to the left and prove the second.}

For (3), first consider a critical vector $\mathbf r$ with
$p_k\le r_i\le2$.  Its reciprocal deficits
\[
 z_i=2\left(\frac{1}{r_i}-\frac{1}{2}\right)
\]
satisfy $0\le z_i\le1/k$ and $\sum_i z_i=1$.  Hence $z$ lies in the convex
hull of the vectors $k^{-1}\mathbf1_S$, $|S|=k$.  At each corresponding
endpoint, \textcolor{black}{first use (2) to move the $k$ copies of $p_k$ to
the first $k$ positions.  The directed rearrangement does not increase the
optimal constant.  Then} the usual Khinchine argument in the complementary
$m-k$ coordinates, followed by the $k$-linear Bohnenblust--Hille inequality,
gives the constant $A_{p_k}^{-(m-k)}B_{\R,k}^{\mult}$.  Mixed-norm
interpolation therefore gives \eqref{eq:Khinchine-interpolation} for
$\mathbf r$.

For a subcritical $\mathbf q$, write
$x_i=2(1/q_i-1/2)$.  Since $x_i\le1/k$ and $\sum_i x_i\le1$, one may choose
$z_i\in[x_i,1/k]$ with $\sum_i z_i=1$; indeed, the available capacity is
$m/k-\sum_i x_i\ge1-\sum_i x_i$.  The critical vector
$r_i=2/(1+z_i)$ satisfies $r_i\le q_i$.  Coordinatewise monotonicity and the
critical case complete the proof.
\end{proof}

\medskip

\begin{lemma}\label{lem:diameter-lower}
{Let $\mathbf q\in\mathcal B_m$, let
$q_{\min}:=\min_{1\le i\le m}q_i$, and set $d=d(\mathbf q)$. Then}
\begin{equation}\label{eq:diameter-lower}
 C_{\mathbf q}^{(m)}\ge
 2^{\frac{(m-1)d}{2(2-d)}}
 \ge \exp\left(\frac{\log2}{4}(m-1)d\right).
\end{equation}
\end{lemma}

\begin{proof}
{The universal estimate \cite[Lemma~1.6]{CostaNunezPellegrinoRaposo2026} gives}
\begin{equation}\label{eq:universal-lower}
 C_{\mathbf q}^{(m)}\ge
 2^{\,1-m+\sum_i 1/q_i+(m-2)/q_{\min}}.
\end{equation}
{Since the remaining exponents are at most $2$,}
\[
 C_{\mathbf q}^{(m)}\ge
 2^{(m-1)(1/q_{\min}-1/2)}.
\]
{Finally $q_{\min}\le2-d$, which gives the first inequality in
\eqref{eq:diameter-lower}; the second follows from $2(2-d)\le4$.}
\end{proof}

\begin{lemma}\label{lem:diameter-upper}
Let $\mathbf q$ be an $m$-linear Bohnenblust--Hille exponent. Write
\[
 q_{\min}:=\min_i q_i,\qquad q_{\max}:=\max_i q_i,
\]
let $d=d(\mathbf q)$, and set
\[
 \varepsilon=d+\frac{2}{m+1}.
\]
Assume $\varepsilon\le2/3$ and define
\begin{equation}\label{eq:optimized-k}
 k=\left\lfloor\frac{2}{\varepsilon}\right\rfloor-1.
\end{equation}
Then $2\le k\le m$, $p_k\le q_{\min}$, and
\begin{equation}\label{eq:diameter-upper-exact}
 C_{\mathbf q}^{(m)}\le
 A_{p_k}^{-(m-k)}B_{\R,k}^{\mult}.
\end{equation}
If $\varepsilon\le1/4$, then
\begin{equation}\label{eq:diameter-upper-explicit}
 C_{\mathbf q}^{(m)}\le
 M\left(\frac{2}{\varepsilon}\right)^\vartheta
 \exp\left(\frac{m\varepsilon}{4(1-\varepsilon)}\right).
\end{equation}
\end{lemma}

\begin{proof}
Since $q_i\le q_{\max}$, admissibility implies
$m/q_{\max}\le(m+1)/2$.  Consequently,
\[
 q_{\min}=q_{\max}-d\ge2-\frac{2}{m+1}-d=2-\varepsilon.
\]
The choice of $k$ gives $2\le k\le m$ and
$2/(k+1)\ge\varepsilon$.  Thus
$p_k=2-2/(k+1)\le2-\varepsilon\le q_{\min}$, and
Proposition~\ref{prop:known-mixed-norm-tools}(3) yields
\eqref{eq:diameter-upper-exact}.

In addition,
\[
 k<\frac{2}{\varepsilon},
 \qquad
 k\ge\frac{2}{\varepsilon}-2=\frac{2(1-\varepsilon)}{\varepsilon}.
\]
If $\varepsilon\le1/4$, Lemma~\ref{lem:elementary-Ap-bound} and the known estimate
\begin{equation}\label{eq:classical-polynomial-bound}
 B_{\R,k}^{\mult}\le Mk^\vartheta
\end{equation}
for an absolute $M\ge1$ \cite[Corollary~3.3]{BayartPellegrinoSeoane2014} give
\[
 C_{\mathbf q}^{(m)}\le
 Mk^\vartheta\exp\left(\frac{m-k}{2k}\right),
\]
and hence \eqref{eq:diameter-upper-explicit}.
\end{proof}

Whenever $\varepsilon\le1/4$, we have therefore proved the two-sided estimate
\begin{equation}\label{eq:diameter-log-envelope}
 \frac{(m-1)d\log2}{2(2-d)}
 \le\log C_{\mathbf q}^{(m)}
 \le\frac{m\varepsilon}{4(1-\varepsilon)}
 +\vartheta\log\frac{2}{\varepsilon}+\log M.
\end{equation}

\subsection{The superpolynomial threshold}\label{subsec:superpoly}

\begin{proof}[Proof of Theorem~\ref{thm:main-A}(i)]
If $m d_m/\log m\to\infty$, then \eqref{eq:diameter-lower} gives
\[
\frac{\bigl(\log C_{\mathbf q^{(m)}}^{(m)}\bigr)}{\log m} \ge\frac{\log2}{4}
\frac{(m-1)d_m}{\log m} \longrightarrow\infty,
\]
which is exactly superpolynomial growth.

Conversely, if $m d_m/\log m$ does not tend to infinity, there are $L>0$
and a subsequence $(m_j)$ such that $m_jd_{m_j}\le L\log m_j$.  Along this
subsequence $d_{m_j}\to0$.  With
$\varepsilon_{m_j}=d_{m_j}+2/(m_j+1)$, we have
{
\[
m_j\varepsilon_{m_j}
=m_jd_{m_j}+\frac{2m_j}{m_j+1}
\le L\log m_j+2,
\]
and $\varepsilon_{m_j}\to0$.  Hence the first term on the right-hand side of
\eqref{eq:diameter-log-envelope} is $O(\log m_j)$.  Moreover,
$\varepsilon_{m_j}\ge 2/(m_j+1)$, so
$\log(2/\varepsilon_{m_j})=O(\log m_j)$, while $\log M=O(1)$.
Therefore
\[
\log C_{\mathbf q^{(m_j)}}^{(m_j)}=O(\log m_j).
\]
Consequently, $C_{\mathbf q^{(m_j)}}^{(m_j)}\le m_j^A$ for some $A>0$
and all sufficiently large $j$, which rules out superpolynomial growth.}
\end{proof}

\subsection{Vanishing diameter}\label{subsec:log-growth}

\begin{proof}[Proof of Theorem~\ref{thm:main-A}(iii)]

The lower estimate is \eqref{eq:diameter-lower}.  For the upper estimate, set
\[
\varepsilon_{m}=d_m+\frac{2}{m+1}, \qquad k_m=\left\lfloor \frac{2}{\varepsilon_{m}}\right\rfloor -1.
\]
{Since $d_m\to0$, we have $\varepsilon_m\to0$, and hence $k_m\to\infty$.  From
\[
 k_m+1\le \frac{2}{\varepsilon_m}<k_m+2
\]
it follows that $k_m\varepsilon_m\to2$.  Moreover,
$m\varepsilon_m=m d_m+O(1)$, and the hypothesis
$m d_m/\log m\to\infty$ implies $m\varepsilon_m\to\infty$; consequently
$k_m/m\to0$.  Finally, $\varepsilon_m/d_m\to1$, because
$2/((m+1)d_m)\to0$ under the same hypothesis.}

{Now \eqref{eq:diameter-upper-exact} gives
\[
 C_{\mathbf q^{(m)}}^{(m)}
 \le A_{p_{k_m}}^{-(m-k_m)}B_{\mathbb R,k_m}^{\mult},
\]
and the growth estimate \eqref{eq:classical-polynomial-bound} yields
\[
\log C_{\mathbf q^{(m)}}^{(m)} \le(m-k_m)(-\log A_{p_{k_m}})
+\vartheta\log k_m+\log M.
\]
Thus the last two terms are negligible on the $m d_m$ scale, since
$k_m=O(1/d_m)\le m$ for all sufficiently large $m$ and hence
$\log k_m=O(\log m)=o(m d_m)$.  For the remaining term,}
\begin{align*}
\frac{(m-k_m)(-\log A_{p_{k_m}})}{md_m}
&=\left(1-\frac{k_m}{m}\right)
\bigl(-k_m\log A_{p_{k_m}}\bigr)\frac{1}{k_md_m}.
\end{align*}
{Here the first factor tends to $1$, Lemma~\ref{lem:Khinchine-asymptotics}
gives $-k_m\log A_{p_{k_m}}\to\vartheta$, and
$k_m d_m=(k_m\varepsilon_m)(d_m/\varepsilon_m)\to2$.  Therefore the last
display tends to $\vartheta/2$, which proves the upper bound and completes
the proof.}
\end{proof}

\subsection{Fixed diameter}\label{subsec:proof-main-B}

\begin{proof}[Proof of Theorem~\ref{thm:main-B-fixed}]
{Put
$\mathbf q_d=(2-d,2,\ldots,2)$ and $t=d/(2-d)\in[0,1]$.  Since
\[
 \frac{1}{2-d}-\frac{1}{2}=\frac{d}{2(2-d)}=\frac{t}{2},
 \qquad\text{equivalently}\qquad
 \frac{1}{2-d}=\frac{1-t}{2}+t,
\]
interpolation between $(2,\ldots,2)$, whose optimal constant is $1$, and
$(1,2,\ldots,2)$, whose optimal constant is $2^{(m-1)/2}$, yields
\begin{equation}\label{eq:fixed-diameter-upper-candidate}
 C_{\mathbf q_d}^{(m)}
 \le \bigl(2^{(m-1)/2}\bigr)^t
 =2^{\frac{(m-1)d}{2(2-d)}}.
\end{equation}
On the other hand, every $\mathbf q\in\mathcal B_m$ with
$d(\mathbf q)=d$ satisfies, by Lemma~\ref{lem:diameter-lower},
\begin{equation}\label{eq:fixed-diameter-lower-all}
 C_{\mathbf q}^{(m)}
 \ge 2^{\frac{(m-1)d(\mathbf q)}{2(2-d(\mathbf q))}}
 =2^{\frac{(m-1)d}{2(2-d)}}.
\end{equation}
Applying \eqref{eq:fixed-diameter-lower-all} to $\mathbf q_d$ and comparing
with \eqref{eq:fixed-diameter-upper-candidate} gives
\[
 C_{\mathbf q_d}^{(m)}=2^{\frac{(m-1)d}{2(2-d)}}.
\]
Hence the minimum in \eqref{eq:intro-fixed-diameter-envelope} has exactly this
value.  \textcolor{black}{Proposition~\ref{prop:known-mixed-norm-tools}(2)
shows that every permutation $\mathbf t$ of $\mathbf q_d$ satisfies
$C_{\mathbf t}^{(m)}\le C_{\mathbf q_d}^{(m)}$.  Since
\eqref{eq:fixed-diameter-lower-all} gives the reverse inequality, every such
permutation has the same optimal constant and is therefore a minimizer.}}

Conversely, suppose equality holds in \eqref{eq:intro-fixed-diameter-envelope}.
{Keeping the intermediate estimate from the proof of
Lemma~\ref{lem:diameter-lower}, we have the full chain
\[
 C_{\mathbf q}^{(m)}
 \ge 2^{\,1-m+\sum_i1/q_i+(m-2)/q_{\min}}
 \ge 2^{(m-1)(1/q_{\min}-1/2)}
 \ge 2^{\frac{(m-1)d}{2(2-d)}}.
\]
Because we are assuming that $\mathbf q$ attains the minimum in \eqref{eq:intro-fixed-diameter-envelope}, the first and last quantities in this chain are equal.  Since the endpoints of this chain coincide, equality must hold at every intermediate step.  If
$q_{i_0}=q_{\min}$, equality in the middle inequality is exactly equality in
\[
 \sum_{i\ne i_0}\frac{1}{q_i}\ge\frac{m-1}{2}.
\]
Since every $q_i\le2$, each summand on the left is at least $1/2$; hence
$q_i=2$ for every $i\ne i_0$.  Consequently $q_{\max}=2$, and
$d=q_{\max}-q_{\min}$ gives $q_{\min}=2-d$.} Hence $\mathbf q$ is a
permutation of $(2-d,2,\ldots,2)$, proving uniqueness up to permutation.

{
For $\mathbf q^{(m)}=(2-d_m,2,\ldots,2)$, the formula just proved gives
\begin{equation}\label{eq:one-coordinate-exact}
 C_{\mathbf q^{(m)}}^{(m)}=2^{\frac{(m-1)d_m}{2(2-d_m)}}.
\end{equation}
Therefore, if $d_m>0$, $d_m\to0$, and $m d_m/\log m\to\infty$, then
\[
 \frac{\log C_{\mathbf q^{(m)}}^{(m)}}{m d_m}
 =\frac{m-1}{m}\frac{\log2}{2(2-d_m)}
 \longrightarrow\frac{\log2}{4},
\]
which proves the optimality assertion.
}
\end{proof}

\section{Exact constants}\label{sec:exact-constants}

If $s=\sigma(\mathbf q)>0$ and
\[
 a=y_1^\downarrow(\mathbf q),\qquad b=y_m^\downarrow(\mathbf q),
\]
then the reconstruction formula gives
\begin{equation}\label{eq:diameter-radial-profile}
 d(\mathbf q)=\frac{2s(a-b)}{(1+sb)(1+sa)}.
\end{equation}

\begin{lemma}\label{lem:universal-profile-lower}
If $\sigma(\mathbf q)=0$, then $C_{\mathbf q}^{(m)}=1$. If $\sigma(\mathbf q)>0$, then
\begin{equation}\label{eq:universal-profile-lower}
 C_{\mathbf q}^{(m)}
 \ge 2^{\frac{\sigma(\mathbf q)}{2}\left(1+(m-2)Y_1(\mathbf q)\right)}.
\end{equation}
\end{lemma}

\begin{proof}
Put $s=\sigma(\mathbf q)>0$. Then
\[
 \sum_{i=1}^m\frac{1}{q_i}=\frac{m}{2}+\frac{s}{2},
 \qquad
 \frac{1}{q_{\min}}=\frac{1}{2}+\frac{sY_1(\mathbf q)}{2}.
\]
Substitution in \eqref{eq:universal-lower} gives \eqref{eq:universal-profile-lower}. The case $s=0$ is the Hilbertian point.
\end{proof}

For mixed-norm interpolation \cite{BenedekPanzone1961,ABPS2014}, let
$\boldsymbol r^{(1)},\ldots,\boldsymbol r^{(N)}\in[1,\infty)^m$ and
$\alpha_\nu\ge0$, $\sum_\nu\alpha_\nu=1$, and define $\boldsymbol r$ by
\[
 \frac{1}{r_i}:=\sum_{\nu=1}^N\frac{\alpha_\nu}{r_i^{(\nu)}},
 \qquad 1\le i\le m.
\]
Iterated H\"older gives
\begin{equation}\label{eq:weighted-mixed-interpolation}
 \|a\|_{\boldsymbol r}
 \le \prod_{\nu=1}^N
 \|a\|_{\boldsymbol r^{(\nu)}}^{\alpha_\nu}.
\end{equation}
\subsection{Extension}\label{subsec:propagation}

Block duplication is an anisotropic version of the recursive
construction in \cite{DinizMunozPellegrinoSeoane2014}.

\begin{lemma}\label{lem:duplication}
Let $\mathbf q=(q_1,\ldots,q_m)$ be a Bohnenblust--Hille exponent, fix $i\in\{1,\ldots,m\}$, and put $\mathbf q^+=(q_1,\ldots,q_m,2)$. Then
\begin{equation}
\label{eq:duplication-lower}C_{\mathbf q^+}^{(m+1)} \ge2^{1/q_i-1/2}C_{\mathbf q}^{(m)}.
\end{equation}

\end{lemma}

\begin{proof}
Fix a finite-dimensional $m$-linear form
\[
S:\ell_{\infty}^{N_{1}}\times\cdots\times\ell_{\infty}^{N_{i}} \times
\cdots\times\ell_{\infty}^{N_{m}}\longrightarrow\R.
\]
Write $x^{(i)}=(x^{(i)\prime},x^{(i)\prime\prime})\in\ell_{\infty}^{2N_{i}}$ with $x^{(i)\prime},x^{(i)\prime\prime}\in\ell_{\infty}^{N_{i}}$, and define
\[
T:\ell_{\infty}^{N_{1}}\times\cdots\times\ell_{\infty}^{2N_{i}} \times
\cdots\times\ell_{\infty}^{N_{m}}\times\ell_{\infty}^{2} \longrightarrow
\R
\]
by
\begin{align*}
T(x^{(1)},\ldots,x^{(m)},y) & =(y_{1}+y_{2})S(\ldots,x^{(i)\prime},\ldots)\\
& \quad+(y_{1}-y_{2})S(\ldots,x^{(i)\prime\prime},\ldots).
\end{align*}
For fixed variables other than $y$, put $A=S(\ldots,x^{(i)\prime},\ldots)$ and $B=S(\ldots,x^{(i)\prime\prime},\ldots)$. Since $A,B\in\R$,
\[
\sup_{\|y\|_{\infty}\le1}|y_{1}(A+B)+y_{2}(A-B)| =|A+B|+|A-B|=2\max
\{|A|,|B|\}.
\]
Taking the supremum over the remaining variables gives $\|T\|=2\|S\|$. At the coefficient level, each coefficient of $S$ is duplicated in the $i$-th index and appears with two signs in the last index. After taking the innermost $\ell_{2}$-norm, each original coefficient is replaced by $\sqrt2$ times its modulus in two disjoint $i$-blocks. Homogeneity of the intervening mixed norms and the $i$-th $\ell_{q_i}$-sum therefore give
\[
\|\widehat T\|_{\mathbf q^+}
 =2^{1/2}2^{1/q_i}\|\widehat S\|_{\mathbf q}.
\]
Hence
\[
 \frac{\|\widehat T\|_{\mathbf q^+}}{\|T\|}
 =2^{1/q_i-1/2}\frac{\|\widehat S\|_{\mathbf q}}{\|S\|}.
\]
Taking suprema in \eqref{eq:finite-dimensional-sup} gives \eqref{eq:duplication-lower}.
\end{proof}

\begin{proposition}\label{prop:exact-extension}
Let $\mathbf q=(q_1,\ldots,q_m)$ be a Bohnenblust--Hille exponent and put $p=\min_i q_i$. For every integer $r\ge0$, set
\[
\mathbf q^{[r]}=(q_1,\ldots,q_m,\underbrace{2,\ldots,2}_{r\text{ times}}).
\]
Then
\begin{equation}
\label{eq:extension-two-sided}2^{r(1/p-1/2)}C_{\mathbf q}^{(m)} \le
C_{\mathbf q^{[r]}}^{(m+r)} \le A_p^{-r}C_{\mathbf q}^{(m)}.
\end{equation}
If $p\le p_{\mathrm H}$, then
\begin{equation}
\label{eq:extension-exact}C_{\mathbf q^{[r]}}^{(m+r)} =2^{r(1/p-1/2)}C_{\mathbf q}^{(m)}.
\end{equation}

\end{proposition}

\begin{proof}
Choose an original coordinate $i$ with $q_i=p$. Appending a coordinate with exponent $2$ leaves the minimum exponent equal to $p$. Lemma~\ref{lem:duplication}, applied at the same original coordinate $i$, therefore gives after one extension
\[
 C_{\mathbf q^{[1]}}^{(m+1)}
 \ge 2^{1/p-1/2}C_{\mathbf q}^{(m)}.
\]
The coordinate carrying the exponent $p$ is still present after the extension, so the argument can be repeated. After $r$ steps,
\[
 C_{\mathbf q^{[r]}}^{(m+r)}
 \ge 2^{r(1/p-1/2)}C_{\mathbf q}^{(m)}.
\]

For the reverse inequality, Proposition~\ref{prop:known-mixed-norm-tools}(1) applies at each step with the same lower Khinchine constant $A_p$, again because the minimum exponent remains $p$. Iteration gives
\[
 C_{\mathbf q^{[r]}}^{(m+r)}
 \le A_p^{-r}C_{\mathbf q}^{(m)}.
\]
These are the two inequalities in \eqref{eq:extension-two-sided}. If
$p\le p_{\mathrm H}$, the first branch of Haagerup's formula gives
$A_p^{-1}=2^{1/p-1/2}$, so the bounds agree at every step and yield
\eqref{eq:extension-exact}.
\end{proof}

\subsection{Two exceptional coordinates}\label{subsec:two-coordinate-data}

For $1\le i<j\le m$ and $a,b\in[1,2]$, let $\mathbf q_{i,j}^{(m)}(a,b)$ be the vector with $q_i=a$, $q_j=b$ and all other coordinates equal to $2$.

\begin{proposition}\label{prop:two-coordinate-exact}
Let $m\ge2$, let $1\le i<j\le m$, and let $a,b\in[1,2]$. Assume
\begin{equation}\label{eq:two-coordinate-admissibility}
 \frac{1}a+\frac{1}b\le\frac{3}{2}.
\end{equation}
Then
\begin{equation}\label{eq:two-coordinate-exact}
 C_{\mathbf q_{i,j}^{(m)}(a,b)}^{(m)}
 =2^{\frac{1}a+\frac{1}b-1+(m-2)\left(\frac{1}{\min\{a,b\}}-\frac{1}{2}\right)}.
\end{equation}
\end{proposition}

\begin{proof}
If $a=b=2$, then $\mathbf q_{i,j}^{(m)}(a,b)=(2,\ldots,2)$ and both sides of \eqref{eq:two-coordinate-exact} are equal to $1$. Otherwise set
\[
 \rho:=2\left(\frac{1}a+\frac{1}b-1\right)
 =\sigma\!\left(\mathbf q_{i,j}^{(m)}(a,b)\right)\in(0,1]
\]
and define $a_*,b_*\in[1,2]$ by
\[
 \frac{1}{a_*}=\frac{1}{2}+\frac{1}{\rho}\left(\frac{1}a-\frac{1}{2}\right),
 \qquad
 \frac{1}{b_*}=\frac{1}{2}+\frac{1}{\rho}\left(\frac{1}b-\frac{1}{2}\right).
\]
Each numerator is nonnegative and no larger than $\rho/2$, so
$1/2\le1/a_*,1/b_*\le1$; in particular $a_*,b_*\in[1,2]$.
Then $1/a_*+1/b_*=3/2$. Since at least one of $a_*,b_*$ is at most
$4/3$, we have $p_*:=\min\{a_*,b_*\}\le4/3<p_{\mathrm H}$.
The sharp bilinear formula \cite[Theorem~6.3]{ABPS2014} therefore gives
\[
 C_{(a_*,b_*)}^{(2)}=2^{1/a_*+1/b_*-1}=\sqrt2.
\]
Append $m-2$ coordinates equal to $2$. Since the minimum remains
$p_*<p_{\mathrm H}$, the exact part of Proposition~\ref{prop:exact-extension} gives
\begin{equation}\label{eq:two-coordinate-critical-first}
 C_{(a_*,b_*,2,\ldots,2)}^{(m)}
 =2^{(m-2)(1/p_*-1/2)}\sqrt2.
\end{equation}
\textcolor{black}{Now place the two exceptional coordinates in the prescribed
positions $i<j$, without changing their left-to-right order. Starting from
$(a_*,b_*,2,\ldots,2)$, this is achieved by repeatedly moving coordinates
equal to $2$ to the left across one of $a_*$ or $b_*$. For every
$c\in[1,2]$, Proposition~\ref{prop:known-mixed-norm-tools}(2) gives}
\[
 \textcolor{black}{C_{(\ldots,2,c,\ldots)}^{(m)}
 \le C_{(\ldots,c,2,\ldots)}^{(m)}.}
\]
\textcolor{black}{Hence none of these moves increases the constant, and
\eqref{eq:two-coordinate-critical-first} gives}
\[
 \textcolor{black}{C_{\mathbf q_{i,j}^{(m)}(a_*,b_*)}^{(m)}
 \le 2^{1/a_*+1/b_*-1+(m-2)(1/p_*-1/2)}.}
\]
\textcolor{black}{The universal lower estimate \eqref{eq:universal-lower}
gives the reverse inequality, since its exponent depends only on the
reciprocal sum and the minimum coordinate. Thus equality holds.}

Finally, by construction,
\[
 \frac{1}a=\frac{1}{2}+\rho\left(\frac{1}{a_*}-\frac{1}{2}\right),\qquad
 \frac{1}b=\frac{1}{2}+\rho\left(\frac{1}{b_*}-\frac{1}{2}\right),
\]
while all remaining coordinates stay equal to $2$. Therefore \eqref{eq:weighted-mixed-interpolation} gives
{
\begin{align*}
 \log_2 C_{\mathbf q_{i,j}^{(m)}(a,b)}^{(m)}
 &\le \rho\left[
 \frac{1}{a_*}+\frac{1}{b_*}-1
 +(m-2)\left(\frac{1}{p_*}-\frac{1}{2}\right)\right]\\
 &=\frac{1}a+\frac{1}b-1
 +(m-2)\left(\frac{1}{\min\{a,b\}}-\frac{1}{2}\right).
\end{align*}}
Indeed, $1/a_*+1/b_*=3/2$ and the definition of $\rho$ gives
$\rho/2=1/a+1/b-1$; moreover, since $p_*=\min\{a_*,b_*\}$,
{
\[
 \rho\left(\frac{1}{p_*}-\frac{1}{2}\right)
 =\max\left\{\frac{1}a-\frac{1}{2},\frac{1}b-\frac{1}{2}\right\}
 =\frac{1}{\min\{a,b\}}-\frac{1}{2}.
\]}
On the other hand, the universal lower estimate \eqref{eq:universal-lower}
applied to $\mathbf q_{i,j}^{(m)}(a,b)$ gives exactly the reverse inequality,
{
\[
 \log_2 C_{\mathbf q_{i,j}^{(m)}(a,b)}^{(m)}
 \ge \frac{1}a+\frac{1}b-1
 +(m-2)\left(\frac{1}{\min\{a,b\}}-\frac{1}{2}\right).
\]}
The two bounds coincide, proving \eqref{eq:two-coordinate-exact}.
\end{proof}

\begin{corollary}\label{cor:two-point-profile-exact}
Let $\mathbf q$ be a Bohnenblust--Hille exponent with $\sigma(\mathbf q)>0$ and with at most two coordinates different from $2$. Then
\begin{equation}\label{eq:two-point-profile-exact}
 C_{\mathbf q}^{(m)}
 =2^{\frac{\sigma(\mathbf q)}{2}\left(1+(m-2)Y_1(\mathbf q)\right)}.
\end{equation}
\end{corollary}

\begin{proof}
Choose positions $i<j$ containing all coordinates that may differ from $2$,
and write the corresponding exponents as $a$ and $b$ (allowing one of them
to equal $2$).  Admissibility gives
\[
 \frac{1}a+\frac{1}b\le\frac{3}{2},
\]
so Proposition~\ref{prop:two-coordinate-exact} applies.  Put
\[
 x_a:=2\left(\frac{1}a-\frac{1}{2}\right),
 \qquad
 x_b:=2\left(\frac{1}b-\frac{1}{2}\right).
\]
\[
 \sigma(\mathbf q)=x_a+x_b,
 \qquad
 \sigma(\mathbf q)Y_1(\mathbf q)=\max\{x_a,x_b\}.
\]
Substitution in \eqref{eq:two-coordinate-exact} gives
\eqref{eq:two-point-profile-exact}.
\end{proof}

Let $X_n$ be a normed
space with distinguished vectors $e_1,\ldots,e_n$, and let $\mathcal T_n$
be a normed class of scalar $m$-linear forms on $X_n^m$.  Assume that
permuting the variables preserves both $\mathcal T_n$ and its norm, and put
\[
 \mathcal K(\mathbf q)
 :=\sup_{n\ge1}\ \sup_{0\ne T\in\mathcal T_n}
 \frac{\|\bigl(T(e_{j_1},\ldots,e_{j_m})\bigr)\|_{\mathbf q}}{\|T\|}
\]
whenever this quantity is finite.

\begin{proposition}
\label{prop:abstract-profile-principle}
\textcolor{black}{On every convex permutation-invariant set
$U\subset(0,1]^m$ of reciprocal exponent vectors for which
$0<\mathcal K(\alpha_1^{-1},\ldots,\alpha_m^{-1})<\infty$ for every
$\boldsymbol\alpha\in U$, the function}
\[
 \textcolor{black}{\Phi(\boldsymbol\alpha)
 :=\log\mathcal K(\alpha_1^{-1},\ldots,\alpha_m^{-1})
 }
\]
\textcolor{black}{is convex. Consequently,}
\[
 \textcolor{black}{\widehat\Phi(\boldsymbol\alpha)
 :=\max_{\pi\in\mathfrak S_m}
 \Phi(\alpha_{\pi(1)},\ldots,\alpha_{\pi(m)})}
\]
\textcolor{black}{is symmetric, convex, and Schur-convex on $U$. Moreover,}
\[
 \textcolor{black}{\widehat\Phi(\boldsymbol\alpha)
 =\Phi(\boldsymbol\alpha^\downarrow).}
\]
\end{proposition}

\begin{proof}
For $\boldsymbol\alpha,\boldsymbol\beta\in U$ and $0\le t\le1$, iterated
H\"older gives
\[
 \|a\|_{((1-t)\boldsymbol\alpha+t\boldsymbol\beta)^{-1}}
 \le
 \|a\|_{\boldsymbol\alpha^{-1}}^{\,1-t}
 \|a\|_{\boldsymbol\beta^{-1}}^{\,t}.
\]
\textcolor{black}{Apply this to the coefficient array of $T$, divide by
$\|T\|$, and take suprema. Thus $\Phi$ is convex. The maximum of its
coordinate permutations is symmetric and convex, hence Schur-convex
\cite[Chapter~3]{MarshallOlkinArnold2011}.}

\textcolor{black}{It remains to identify the maximizing order. Write
$q_i=\alpha_i^{-1}$. If $q_i\ge q_{i+1}$, Minkowski's inequality gives}
\[
 \textcolor{black}{\|a\|_{(\ldots,q_i,q_{i+1},\ldots)}
 \le
 \|a^{(i\,i+1)}\|_{(\ldots,q_{i+1},q_i,\ldots)},}
\]
\textcolor{black}{where the two corresponding indices of $a$ are
interchanged. The array on the right belongs to the form obtained by
permuting the same variables, whose norm is unchanged. Successive adjacent
transpositions place the smallest exponents, or equivalently the largest
reciprocals, in the outermost sums. This proves the last identity.}
\end{proof}

\begin{proof}[Proof of \textcolor{black}{Theorem}~\ref{thm:main-C}]

\textcolor{black}{Apply Proposition~\ref{prop:abstract-profile-principle} to
the class of all real $m$-linear forms on $(\ell_\infty^n)^m$. By zero
extension from rectangular coordinate blocks, the resulting
$\mathcal K(\mathbf q)$ is $C_{\mathbf q}^{(m)}$. On the fixed-$s$ slice,}
\[
 \textcolor{black}{\frac{1}{Q_i(s,z)}=\frac{1}{2}+\frac{s z_i}{2}}
\]
\textcolor{black}{depends affinely on $z$. The proposition therefore gives
the asserted convexity and Schur-convexity, as well as
\eqref{eq:intro-canonical-order} and
\eqref{eq:intro-actual-majorization}. Finally,}
\[
 \textcolor{black}{(1/m,\ldots,1/m)\prec z\prec(1,0,\ldots,0).}
\]
\textcolor{black}{The concentrated profile has constant $2^{(m-1)s/2}$ by
the exact one-coordinate formula. Hence
\eqref{eq:intro-profile-extremes} follows.}
\end{proof}

\subsection{Majorization}\label{subsec:majorization-exactness}

\begin{proposition}\label{prop:two-point-majorant}
Let $\mathbf q$ be a Bohnenblust--Hille exponent with
$s:=\sigma(\mathbf q)>0$.  If
\[
 \frac{1}{2}\le a\le1,
 \qquad Y_1(\mathbf q)\le a,
\]
then
\begin{equation}\label{eq:two-point-majorant}
 C_{\mathbf q}^{(m)}
 \le 2^{\frac{s}{2}(1+(m-2)a)}.
\end{equation}
\end{proposition}

\begin{proof}
Set $z=(a,1-a,0,\ldots,0)$.  Since $a\ge1/2$, this vector is decreasing.
The first ordered partial sum of $y(\mathbf q)$ is at most $a$, and for
$2\le r<m$ both $X_r(y(\mathbf q))$ and $X_r(z)$ are bounded by $1$, with
$X_r(z)=1$.  Hence $y(\mathbf q)\prec z$.

By \textcolor{black}{Rado's theorem~\cite{Rado1952}}, there are permutations $\pi_\nu$ and weights
$\lambda_\nu\ge0$, $\sum_\nu\lambda_\nu=1$, such that
\[
 y(\mathbf q)=\sum_\nu\lambda_\nu\pi_\nu z.
\]
Define $\mathbf r^{(\nu)}$ by
\[
 \frac{1}{r_i^{(\nu)}}=\frac{1}{2}+\frac{s}{2}(\pi_\nu z)_i.
\]
Each $\mathbf r^{(\nu)}$ has at most two non-Hilbertian coordinates, total
deficit $s$, and largest profile mass $a$.  By
Corollary~\ref{cor:two-point-profile-exact},
\[
 C_{\mathbf r^{(\nu)}}^{(m)}
 =2^{\frac{s}{2}(1+(m-2)a)}.
\]
Equation~\eqref{eq:weighted-mixed-interpolation} gives \eqref{eq:two-point-majorant}.
\end{proof}

\begin{proof}[Proof of Theorem~\ref{thm:main-D}(i)]
For $s=0$ there is nothing to prove.  If $s>0$, apply
Proposition~\ref{prop:two-point-majorant} with $a=1$ to obtain
\[
 C_{\mathbf q}^{(m)}\le2^{(m-1)s/2}
\]
for every exponent vector of total deficit $s$.  The vector
$\mathbf q_s=(2/(1+s),2,\ldots,2)$ has total deficit $s$ and diameter
$2s/(1+s)$.  Hence Theorem~\ref{thm:main-B-fixed} gives
\[
 C_{\mathbf q_s}^{(m)}=2^{(m-1)s/2}.
\]
The same is true of every permutation of $\mathbf q_s$.
\end{proof}

\begin{proof}[Proof of Theorem~\ref{thm:main-D}(ii)]

Since $Y_1(\mathbf q)\ge1/2$, Proposition~\ref{prop:two-point-majorant} applies with $a=Y_1(\mathbf q)$ and gives
\[
 C_{\mathbf q}^{(m)}\le 2^{\frac{\sigma(\mathbf q)}{2}(1+(m-2)Y_1(\mathbf q))}.
\]
Lemma~\ref{lem:universal-profile-lower} gives exactly the reverse inequality, so the two bounds coincide.
\end{proof}

\begin{proposition}\label{prop:global-explicit-profile-corridor}
Let $\mathbf q$ be a Bohnenblust--Hille exponent with $\sigma(\mathbf q)>0$. Then
\[
 2^{\frac{\sigma(\mathbf q)}{2}\left(1+(m-2)Y_1(\mathbf q)\right)}
 \le C_{\mathbf q}^{(m)}
 \le
 2^{\frac{\sigma(\mathbf q)}{2}\left(1+(m-2)\max\{Y_1(\mathbf q),1/2\}\right)}.
\]
In particular, if $Y_1(\mathbf q)<1/2$, then
\[
 2^{\frac{\sigma(\mathbf q)}{2}\left(1+(m-2)Y_1(\mathbf q)\right)}
 \le C_{\mathbf q}^{(m)}
 \le 2^{\sigma(\mathbf q)m/4}.
\]
If $Y_1(\mathbf q)\ge1/2$, the two sides coincide.
\end{proposition}

\begin{proof}
The lower bound is Lemma~\ref{lem:universal-profile-lower}.  For the upper
bound, apply Proposition~\ref{prop:two-point-majorant} with
\[
 a=\max\left\{Y_1(\mathbf q),\frac{1}{2}\right\}.
\]
If $Y_1(\mathbf q)<1/2$, the exponent becomes
$\sigma(\mathbf q)m/4$; if $Y_1(\mathbf q)\ge1/2$, it agrees with the lower
exponent.
\end{proof}

\begin{proof}[Proof of Theorem~\ref{thm:main-A}(ii)]

The lower estimate follows from \eqref{eq:diameter-lower}:
\[
 \frac{\bigl(\log C_{\mathbf q^{(m)}}^{(m)}\bigr)}{m d_m}
 \ge \left(1-\frac{1}m\right)\frac{\log2}{2(2-d_m)}
 \ge \left(1-\frac{1}m\right)\frac{\log2}{4}.
\]
For the upper estimate, first note a bound valid throughout $\mathcal B_m$.
Fix $\mathbf q\in\mathcal B_m$ and put
\[
 x_i:=2\left(\frac{1}{q_i}-\frac{1}{2}\right),\qquad 1\le i\le m,
 \qquad s:=\sum_{i=1}^m x_i\le1.
\]
For each $i$, let $\mathbf r^{(i)}$ be the exponent vector whose $i$th coordinate is $1$ and whose remaining coordinates are $2$. \textcolor{black}{By the directed rearrangement in Proposition~\ref{prop:known-mixed-norm-tools}(2) and \cite[Theorem~2.1]{PellegrinoJNT2016},}
\[
 C_{\mathbf r^{(i)}}^{(m)}\le 2^{(m-1)/2},
 \qquad 1\le i\le m.
\]
Moreover, coordinate by coordinate,
\[
 \frac{1}{q_j}
 =(1-s)\frac{1}{2}+\sum_{i=1}^m x_i\frac{1}{r_j^{(i)}},
 \qquad 1\le j\le m.
\]
Mixed-norm interpolation, together with
$C_{(2,\ldots,2)}^{(m)}=1$, therefore yields
\[
 C_{\mathbf q}^{(m)}
 \le\prod_{i=1}^m\left(C_{\mathbf r^{(i)}}^{(m)}\right)^{x_i}
 \le 2^{(m-1)s/2}
 \le 2^{(m-1)/2}.
\]
Thus
\begin{equation}\label{eq:global-exponential-upper}
 C_{\mathbf q}^{(m)}\le 2^{(m-1)/2}.
\end{equation}

Set $\varepsilon_m=d_m+2/(m+1)$; then $\varepsilon_m/d_m\to1$.  On the
indices for which $\varepsilon_m\le2/3$, let $k_m$ be given by
\eqref{eq:optimized-k}.  Since
\[
 k_m+1\le\frac{2}{\varepsilon_m}<k_m+2,
\]
\eqref{eq:diameter-upper-exact},
\eqref{eq:classical-polynomial-bound}, and
\eqref{eq:Ap-pk-dyadic-bound} give
\[
 \log C_{\mathbf q^{(m)}}^{(m)}
 \le \frac{(m-k_m)\log2}{k_m+2}
 +\vartheta\log k_m+\log M
 \le \frac{m\varepsilon_m\log2}{2}+O(\log m).
\]
After division by $m d_m$, the error tends to zero and the resulting
limsup is at most $(\log2)/2$.

For $\varepsilon_m>2/3$, apply Proposition~\ref{prop:global-explicit-profile-corridor}. Write
$s_m=\sigma(\mathbf q^{(m)})$ and $a_m=Y_1(\mathbf q^{(m)})$.  If
$a_m<1/2$, then
\[
 \log C_{\mathbf q^{(m)}}^{(m)}
 \le \frac{s_m m\log2}{4}\le\frac{m\log2}{4},
\]
and hence the normalized ratio is at most $3\log2/8+o(1)$.  If
$a_m\ge1/2$, Theorem~\ref{thm:main-D}(ii) gives
\[
 \frac{\log C_{\mathbf q^{(m)}}^{(m)}}{m d_m}
 =\frac{\log2}{2d_m}
 \left(\frac{s_m}{m}+\left(1-\frac{2}m\right)s_ma_m\right).
\]
Now $s_ma_m=2(1/q_{\min}^{(m)}-1/2)$, while admissibility gives
$q_{\max}^{(m)}\ge2m/(m+1)$ and therefore
\[
 q_{\min}^{(m)}=q_{\max}^{(m)}-d_m
 \ge\frac{2m}{m+1}-d_m.
\]
\textcolor{black}{If a subsequence with $\varepsilon_m\le2/3$ realizes the
global limsup, the preceding estimate suffices. Otherwise choose a subsequence
in $\{m:\varepsilon_m>2/3\}$ realizing the global limsup; after a further
extraction, assume $d_m\to d$ for some $d\in[2/3,1]$.}  The preceding
estimate then gives
\[
 \limsup_{m\to\infty}\frac{\log C_{\mathbf q^{(m)}}^{(m)}}{m d_m}
 \le\frac{\log2}{2(2-d)}
 \le\frac{\log2}{2}.
\]
Since the same conclusion holds for every subsequence realizing the limsup, this proves the required global upper bound and hence \eqref{eq:intro-global-growth-law}.  The last coefficient is
sharp: for $\mathbf q^{(m)}=(1,2,\ldots,2)$, the normalized ratio tends to
$(\log2)/2$.
\end{proof}

\section{Profile asymptotics}\label{sec:proof-main-E}

For $1\le k\le m$, set
\[
\mathbf q^{(m,k)}=(\underbrace{p_k,\ldots,p_k}_{k}, \underbrace
{2,\ldots,2}_{m-k}).
\]
Then $C_{m,k}=C_{\mathbf q^{(m,k)}}^{(m)}$ by
\eqref{eq:intro-flat-endpoint-constant}.

The extension estimate gives, for every $m\ge2$ and $1\le k\le m$,
\begin{equation}
\label{eq:Cmk-two-sided}2^{\frac{m-k}{2k}}B_{\R,k}^{\mult} \le
C_{m,k} \le A_{p_k}^{-(m-k)}B_{\R,k}^{\mult}.
\end{equation}
Indeed, for $k=1$, Equation~\eqref{eq:one-coordinate-exact} gives
$C_{m,1}=2^{(m-1)/2}$, which has the asserted form because
$B_{\R,1}^{\mult}=1$.  For $k\ge2$, apply
Proposition~\ref{prop:exact-extension} to the diagonal $k$-linear exponent
$(p_k,\ldots,p_k)$.  Since $1/p_k-1/2=1/(2k)$,
\eqref{eq:Cmk-two-sided} follows.  Finally,
\[
 p_{12}=\frac{24}{13}=1.84615\ldots
 <p_{\mathrm H}=1.847416\ldots
 <\frac{13}{7}=1.85714\ldots=p_{13}.
\]
For $1\le k\le12$, $A_{p_k}^{-1}=2^{1/(2k)}$, and \eqref{eq:Cmk-two-sided} gives
\begin{equation}
\label{eq:Cmk-exact}
 C_{m,k}=2^{\frac{m-k}{2k}}B_{\R,k}^{\mult},\qquad
 C_{m+1,k}=2^{1/(2k)}C_{m,k},\qquad 1\le k\le12.
\end{equation}

\begin{proof}[Proof of Theorem~\ref{thm:main-E}]
Let $(\mathbf q^{(m)})_{m\ge2}$ satisfy its assumptions, and put
\[
 d_m=d(\mathbf q^{(m)}),\qquad
 s_m=\sigma(\mathbf q^{(m)}).
\]
For all sufficiently large $m$, write the ordered normalized profile as
\[
 y_{1,m}\ge y_{2,m}\ge\cdots\ge y_{m,m}\ge0,
 \qquad
 \sum_{j=1}^m y_{j,m}=1,
\]
where $y_{j,m}=y_j^\downarrow(\mathbf q^{(m)})$, and set $y_{m+1,m}=0$.
For $j\ge13$, with $p_j$ as in \eqref{eq:intro-pk-definition}, define the Khinchine propagation error
\begin{equation}\label{eq:endpoint-defect-delta}
 \delta_j
 :=j\bigl(-\log A_{p_j}\bigr)-\frac{\log2}{2},
 \qquad j\ge13,
\end{equation}
Then
\begin{equation}\label{eq:multiscale-defect-upper}
 \limsup_{m\to\infty}
 \frac{\bigl(\log C_{\mathbf q^{(m)}}^{(m)}\bigr)}{m d_m}
 \le
 \frac{\log2}{4}
 +\limsup_{m\to\infty}
 \frac{1}{2y_{1,m}}
 \sum_{j=13}^m
 \delta_j\bigl(y_{j,m}-y_{j+1,m}\bigr).
\end{equation}
For the diameter normalization, no lower bound on $y_{1,m}$ is needed.  Put
\[
 a_m:=y_{1,m},\qquad b_m:=y_{m,m}.
\]
Since the profile is nonnegative and has total mass one, $b_m\le1/m$. The exact diameter formula gives
\begin{equation}\label{eq:multiscale-diameter-exact}
 d_m=
 \frac{2s_m(a_m-b_m)}{(1+s_m b_m)(1+s_m a_m)}.
\end{equation}
In particular, $d_m\le2(a_m-b_m)$. Hence
\[
 \frac{m(a_m-b_m)}{\log m}
 \ge\frac{m d_m}{2\log m}\longrightarrow\infty.
\]
Since $b_m\le1/m$, it follows that
\[
 \frac{b_m}{a_m-b_m}\longrightarrow0,
 \qquad
 \frac{b_m}{a_m}\longrightarrow0.
\]
Since the denominator in \eqref{eq:multiscale-diameter-exact} is at most
$4$,
\[
 d_m\ge\frac{s_m a_m}{2}\left(1-\frac{b_m}{a_m}\right).
\]
Thus $d_m\to0$ and $b_m/a_m\to0$ imply $s_m a_m\to0$.  Since also
$s_m b_m\le s_m a_m$, we obtain
\begin{equation}\label{eq:multiscale-diameter-equivalence}
 d_m=2s_m a_m(1+o(1)).
\end{equation}

For $S\subset[m]$ with $|S|=j$, set
\[
 u_{j,S}:=\frac1j\mathbf1_S,
 \qquad
 t_{j,S,i}:=\begin{cases}p_j,&i\in S,\\2,&i\notin S,\end{cases}
 \qquad
 \mathbf t_{j,S}:=(t_{j,S,1},\ldots,t_{j,S,m}).
\]
Then $x(\mathbf t_{j,S})=u_{j,S}$ and $\sigma(\mathbf t_{j,S})=1$.

Choose nested sets
\[
 S_{1,m}\subset S_{2,m}\subset\cdots\subset S_{m,m}=[m]
\]
so that $S_{j,m}$ contains the indices of the $j$ largest coordinates of
$y(\mathbf q^{(m)})$; ties may be resolved arbitrarily.  Define
\[
 \lambda_{j,m}:=j\bigl(y_{j,m}-y_{j+1,m}\bigr),
 \qquad 1\le j\le m.
\]
Then $\lambda_{j,m}\ge0$, $\sum_{j=1}^m\lambda_{j,m}=1$, and
\begin{equation}\label{eq:canonical-level-decomposition}
 y(\mathbf q^{(m)})
 =\sum_{j=1}^m\lambda_{j,m}u_{j,S_{j,m}}.
\end{equation}
For a coordinate of rank $i$,
\[
 \sum_{j=i}^m\frac{\lambda_{j,m}}{j}
 =\sum_{j=i}^m(y_{j,m}-y_{j+1,m})
 =y_{i,m}.
\]

For $1\le j\le m$, define
\[
 h_j:=
 \begin{cases}
 \dfrac{\log2}{2j},&1\le j\le12,\\[2mm]
 -\log A_{p_j},&j\ge13.
 \end{cases}
\]
\textcolor{black}{For $j\le12$, the directed rearrangement and
\eqref{eq:Cmk-exact} give}
\[
 \textcolor{black}{C_{\mathbf t_{j,S_{j,m}}}^{(m)}
 \le C_{m,j}=2^{\frac{m-j}{2j}}B_{\R,j}^{\mult}.}
\]
\textcolor{black}{For $j\ge13$, the same rearrangement and
\eqref{eq:Khinchine-interpolation} give}
\[
 \textcolor{black}{C_{\mathbf t_{j,S_{j,m}}}^{(m)}
 \le C_{m,j}\le A_{p_j}^{-(m-j)}B_{\R,j}^{\mult}.}
\]
Define, for $1\le j\le m$,
\[
 U_{m,j}:=
 \begin{cases}
 2^{\frac{m-j}{2j}}B_{\R,j}^{\mult},&1\le j\le12,\\[2mm]
 A_{p_j}^{-(m-j)}B_{\R,j}^{\mult},&j\ge13.
 \end{cases}
\]
Then $C_{\mathbf t_{j,S_{j,m}}}^{(m)}\le U_{m,j}$ and
\begin{equation}\label{eq:multiscale-endpoint-bound}
 \log U_{m,j}=(m-j)h_j+\log B_{\R,j}^{\mult}.
\end{equation}
Moreover, since $B_{\R,j}^{\mult}\le Mj^\vartheta$ and
$\sum_j\lambda_{j,m}=1$,
\[
 \sum_{j=1}^m\lambda_{j,m}\log B_{\R,j}^{\mult}
 \le \log M+\vartheta\log m=O(\log m),
\]
with an absolute implicit constant.

Since $x(2,\ldots,2)=0$, \eqref{eq:canonical-level-decomposition} gives
\[
 x(\mathbf q^{(m)})
 =s_m\sum_{j=1}^m\lambda_{j,m}x(\mathbf t_{j,S_{j,m}}).
\]
The reciprocal-coordinate interpolation inequality therefore gives
\begin{align}
 \log C_{\mathbf q^{(m)}}^{(m)}
 &\le s_m\sum_{j=1}^m\lambda_{j,m}\log U_{m,j}\notag\\
 &\le s_m m\sum_{j=1}^m\lambda_{j,m}h_j
 +O(s_m\log m),                   \label{eq:multiscale-before-telescope}
\end{align}
where we have discarded the nonpositive terms
$-s_m\lambda_{j,m}j h_j$; here $h_j\ge0$ for every $j$, since
$A_{p_j}\le1$.

For $j\le12$ one has $j h_j=\log2/2$, while for $j\ge13$ equation \eqref{eq:endpoint-defect-delta} gives $j h_j=\log2/2+\delta_j$. Hence
\begin{align}
 \sum_{j=1}^m\lambda_{j,m}h_j
 &=\sum_{j=1}^m
 (y_{j,m}-y_{j+1,m})j h_j\notag\\
 &=\frac{\log2}{2}
 \sum_{j=1}^m(y_{j,m}-y_{j+1,m})
 +\sum_{j=13}^m
 \delta_j(y_{j,m}-y_{j+1,m})\notag\\
 &=\frac{\log2}{2}a_m
 +\sum_{j=13}^m
 \delta_j(y_{j,m}-y_{j+1,m}).         \label{eq:multiscale-telescope}
\end{align}
Equations~\eqref{eq:multiscale-before-telescope}--\eqref{eq:multiscale-diameter-equivalence} give
\[
 \limsup_{m\to\infty}
 \frac{\bigl(\log C_{\mathbf q^{(m)}}^{(m)}\bigr)}{m d_m}
 \le
 \frac{\log2}{4}
 +\limsup_{m\to\infty}
 \frac{1}{2a_m}
 \sum_{j=13}^m
 \delta_j(y_{j,m}-y_{j+1,m}),
\]
because
\[
 \frac{s_m\log m}{m d_m}
 =\frac{\log m}{2ma_m}(1+o(1))\longrightarrow0.
\]
Indeed, $b_m/a_m\to0$ gives
$a_m-b_m=a_m(1+o(1))$, and therefore
\[
 \frac{m a_m}{\log m}
 =\frac{m(a_m-b_m)}{\log m}(1+o(1))\longrightarrow\infty.
\]
Proposition~\ref{prop:exact-extension} gives
\[
 A_{p_j}^{-1}\ge2^{1/p_j-1/2}=2^{1/(2j)},
\]
so $\delta_j\ge0$.  For the upper bound, put $z=1/(j+1)$ and
\[
 H(z):=\log\Gamma\!\left(\frac{3}{2}-z\right)-\frac{1}{2}\log\pi.
\]
Haagerup's second branch gives the exact identity
\begin{equation}\label{eq:defect-secant-identity}
 \delta_j=\frac{H(0)-H(z)}{2z}.
\end{equation}
The function $H$ is convex, because
$H''(z)=\psi'(3/2-z)>0$.  Therefore
\[
 H(z)\ge H(0)+H'(0)z,
\]
and \eqref{eq:defect-secant-identity} yields
\begin{equation}\label{eq:endpoint-defect-uniform-bound}
 0\le \delta_j\le-\frac{H'(0)}{2}
 =\frac{\psi(3/2)}{2}
 =\vartheta-\frac{\log2}{2},
 \qquad j\ge13.
\end{equation}
Letting $j\to\infty$ in \eqref{eq:defect-secant-identity} gives
\[
 \delta_j\longrightarrow\vartheta-\frac{\log2}{2};
\]
in fact, convexity also shows that $\delta_j$ increases to this limit.

Since
\[
 \sum_{j=13}^m(y_{j,m}-y_{j+1,m})=y_{13,m},
\]
we obtain from \eqref{eq:endpoint-defect-uniform-bound}
\[
 \limsup_{m\to\infty}
 \frac{\bigl(\log C_{\mathbf q^{(m)}}^{(m)}\bigr)}{m d_m}
 \le
 \frac{\log2}{4}
 +\left(\frac{\vartheta}{2}-\frac{\log2}{4}\right)
 \limsup_{m\to\infty}\frac{y_{13,m}}{a_m}.
\]
Since $a_m=y_{1,m}=Y_1(\mathbf q^{(m)})$, this is \eqref{eq:intro-main-defect-tail-bound}.
\end{proof}

Let $(\mathbf q^{(m)})_{m\ge2}$ satisfy the assumptions of Theorem~\ref{thm:main-E}, and put
\[
 d_m=d(\mathbf q^{(m)}),\qquad
 s_m=\sigma(\mathbf q^{(m)}),\qquad
 a_m=Y_1(\mathbf q^{(m)}).
\]
Assume, in addition, that $a_m\ge1/12$ for all sufficiently large $m$.
For all sufficiently large $m$, choose $k_m\in\{1,\ldots,11\}$ and $\theta_m\in[0,1]$ such that
\[
 \frac{1}{k_m+1}\le a_m\le\frac{1}{k_m},
 \qquad
 a_m=\frac{\theta_m}{k_m}+\frac{1-\theta_m}{k_m+1}.
\]
Let $z^{(m)}$ be the decreasing probability vector whose first $k_m$ coordinates equal $a_m$, whose $(k_m+1)$-st coordinate equals $(1-\theta_m)/(k_m+1)$, and whose remaining coordinates are zero.  Equivalently,
\[
 z^{(m)}=\theta_m\frac1{k_m}\mathbf1_{\{1,\ldots,k_m\}}
 +(1-\theta_m)\frac1{k_m+1}\mathbf1_{\{1,\ldots,k_m+1\}}.
\]
For $1\le r\le k_m$,
\[
 X_r(y(\mathbf q^{(m)}))\le r a_m=X_r(z^{(m)}),
\]
and for $r\ge k_m+1$ both partial sums are bounded by, and $X_r(z^{(m)})$ equals, $1$.  Hence
\[
 y(\mathbf q^{(m)})\prec z^{(m)}.
\]
Define $\mathbf r^{(m)}$ by
\[
 \frac1{r_i^{(m)}}=\frac12+\frac{s_m}{2}z_i^{(m)},\qquad 1\le i\le m.
\]
\textcolor{black}{The definition of the canonical profile constant and
Theorem~\ref{thm:main-C} give}
\[
 \textcolor{black}{C_{\mathbf q^{(m)}}^{(m)}
 \le \mathfrak C_{m,s_m}(y(\mathbf q^{(m)}))
 \le \mathfrak C_{m,s_m}(z^{(m)})
 = C_{\mathbf r^{(m)}}^{(m)}.}
\]
Since $k_m,k_m+1\le12$, \eqref{eq:Cmk-exact} and reciprocal-coordinate interpolation, with weights $s_m\theta_m$, $s_m(1-\theta_m)$ and $1-s_m$, give
\begin{align*}
 \log C_{\mathbf q^{(m)}}^{(m)}
 &\le \frac{s_m\log2}{2}
 \left[
 \theta_m\frac{m-k_m}{k_m}
 +(1-\theta_m)\frac{m-k_m-1}{k_m+1}
 \right]\\
 &\quad+s_m\theta_m\log B_{\R,k_m}^{\mult}
 +s_m(1-\theta_m)\log B_{\R,k_m+1}^{\mult}.
\end{align*}
Since $k_m\le11$ and $B_{\R,k}^{\mult}$ is fixed for $1\le k\le12$, the last two terms are $O(s_m)$. Moreover,
\begin{align*}
 \theta_m\frac{m-k_m}{k_m}
 +(1-\theta_m)\frac{m-k_m-1}{k_m+1}
 &=m\left(\frac{\theta_m}{k_m}+\frac{1-\theta_m}{k_m+1}\right)-1\\
 &=m a_m-1.
\end{align*}
Thus
\begin{equation}\label{eq:twelve-level-upper-open}
 \log C_{\mathbf q^{(m)}}^{(m)}
 \le \frac{s_m(m a_m-1)\log2}{2}+O(s_m)
 =\frac{s_m m a_m\log2}{2}+O(s_m).
\end{equation}
For the diameter normalization, put
\[
 b_m:=y_m^\downarrow(\mathbf q^{(m)}).
\]
Since $a_m\ge1/12$ and the normalized profile has total mass one,
\[
 b_m\le \frac{1-a_m}{m-1}\le \frac{1}{m-1},
 \qquad
 \frac{b_m}{a_m}\le \frac{12}{m-1}\longrightarrow0.
\]
By the exact diameter formula \eqref{eq:diameter-radial-profile},
\[
 d_m=
 \frac{2s_m(a_m-b_m)}{(1+s_m b_m)(1+s_m a_m)}.
\]
Because $a_m-b_m$ is bounded away from zero for all sufficiently large $m$ and $d_m\to0$, it follows that $s_m\to0$. Therefore
\[
 d_m=2s_m a_m(1+o(1)).
\]
Because $a_m\ge1/12$, the error term in \eqref{eq:twelve-level-upper-open} satisfies
\[
 \frac{O(s_m)}{m d_m}=O\!\left(\frac{1}{m a_m}\right)\longrightarrow0.
\]
Consequently,
\[
 \limsup_{m\to\infty}
 \frac{\bigl(\log C_{\mathbf q^{(m)}}^{(m)}\bigr)}{m d_m}
 \le\frac{\log2}{4}.
\]
The reverse inequality is the universal lower coefficient from Theorem~\ref{thm:main-A}(iii).  Therefore
\[
 \frac{\log C_{\mathbf q^{(m)}}^{(m)}}{m d_m}
 \longrightarrow\frac{\log2}{4},
\]
which proves case \textup{(a)} of Corollary~\ref{cor:intro-mainD-collapse}.

For case \textup{(b)}, Theorem~\ref{thm:main-E} gives
\[
 \limsup_{m\to\infty}
 \frac{\log C_{\mathbf q^{(m)}}^{(m)}}{m d_m}
 \le \frac{\log2}{4}
\]
when $y_{13,m}/y_{1,m}\to0$, while Theorem~\ref{thm:main-A}(iii) gives the
matching lower bound.  Thus the normalized logarithm converges to
$(\log2)/4$, as claimed.

\section*{Statements and declarations}
\textcolor{black}{\noindent\textbf{Funding.} D. N\'u\~nez-Alarc\'on and
D. M. Pellegrino are supported by CNPq Grants No.~406457/2023-9
(CNPq/MCTI Call No.~10/2023) and No.~403964/2024-5 (MCTI/CNPq Call
No.~16/2024), both from the Conselho Nacional de Desenvolvimento
Cient\'ifico e Tecnol\'ogico (CNPq, Brazil). D. M. Pellegrino is also
supported by CNPq Grant No.~305807/2025-0. E. V. Teixeira gratefully
acknowledges support from the Grayce B. Kerr Chair funds at Oklahoma State
University.}

\textcolor{black}{Part of this work was carried out under the DARPA ExpMath
project \emph{``A Human-Centered Framework for AI-Mathematician
Collaboration in Research-Level Mathematics''} (Agreement
No.~HR0011262E029), in which E. V. Teixeira serves as a co-principal
investigator and gratefully acknowledges partial support. The views and
conclusions expressed here are those of the authors and should not be
interpreted as representing the official policies of the Department of
Defense or the U.S. Government.}

\textcolor{black}{The authors thank Jaume de Dios, Chinmay Hegde, Claudio
Silva, and Samuel Westrick for helpful discussions related to the ExpMath
collaboration.}

\medskip
\noindent\textbf{Competing interests.} The authors declare that they have no competing interests.

\noindent\textbf{Data availability.} No datasets were generated or analyzed in the course of this work.

\medskip
\noindent\textbf{Machine assistance.} Generative AI tools were used for language editing, bibliographic searches, consistency checks, and routine \LaTeX{} preparation.  All mathematical statements, proofs, and citations were independently checked by the authors, who take full responsibility for the manuscript.


\begin{thebibliography}{99}
\bibitem {ABPS2014}N.~Albuquerque, F.~Bayart, D.~Pellegrino and
J.~B.~Seoane-Sep\'ulveda, \emph{Sharp generalizations of the multilinear
Bohnenblust--Hille inequality}, J. Funct. Anal. \textbf{266} (2014), no.~6,
3726--3740,
\href{https://doi.org/10.1016/j.jfa.2013.08.013}{doi:10.1016/j.jfa.2013.08.013}.
\bibitem {AraujoPellegrino2019}G.~Ara\'ujo and D.~Pellegrino,
\emph{A Gale--Berlekamp permutation-switching problem in higher dimensions},
European J. Combin. \textbf{77} (2019), 17--30,
\href{https://doi.org/10.1016/j.ejc.2018.10.007}{doi:10.1016/j.ejc.2018.10.007}.
\bibitem {ArunachalamDuttEscuderoPalazuelos2025}
S.~Arunachalam, A.~Dutt, F.~Escudero Guti\'errez and C.~Palazuelos,
\emph{A cb-Bohnenblust--Hille inequality with constant one and its applications in learning theory},
Math. Ann. \textbf{392} (2025), no.~3, 3367--3396,
\href{https://doi.org/10.1007/s00208-025-03142-5}{doi:10.1007/s00208-025-03142-5}.
\bibitem {BayartPellegrinoSeoane2014}F.~Bayart, D.~Pellegrino and
J.~B.~Seoane-Sep\'ulveda, \emph{The Bohr radius of the $n$-dimensional
polydisk is equivalent to $\sqrt{(\log n)/n}$}, Adv. Math. \textbf{264}
(2014), 726--746,
\href{https://doi.org/10.1016/j.aim.2014.07.029}{doi:10.1016/j.aim.2014.07.029}.
\bibitem {BenedekPanzone1961}A.~Benedek and R.~Panzone,
\emph{The space $L^p$, with mixed norm}, Duke Math. J. \textbf{28} (1961),
no.~3, 301--324,
\href{https://doi.org/10.1215/S0012-7094-61-02828-9}{doi:10.1215/S0012-7094-61-02828-9}.
\bibitem {BohnenblustHille1931}H.~F. Bohnenblust and E.~Hille,
\emph{On the absolute convergence of Dirichlet series}, Ann. of Math. (2)
\textbf{32} (1931), no.~3, 600--622,
\href{https://doi.org/10.2307/1968255}{doi:10.2307/1968255}.
\bibitem {CostaNunezPellegrinoRaposo2026}F.~Costa Jr., D.~N\'u\~nez-Alarc\'on,
D.~Pellegrino and A.~Raposo Jr., \emph{Characterizing the growth of the
anisotropic Bohnenblust--Hille constants}, J. Math. Anal. Appl. \textbf{556}
(2026), no.~1, Paper No.~130220, 18~pp.,
\href{https://doi.org/10.1016/j.jmaa.2025.130220}{doi:10.1016/j.jmaa.2025.130220}.
\bibitem {DinizMunozPellegrinoSeoane2014}D.~Diniz, G.~A.~Mu\~noz-Fern\'andez, D.~Pellegrino and
J.~B.~Seoane-Sep\'ulveda, \emph{Lower bounds for the constants in the Bohnenblust--Hille inequality: the case of real scalars}, Proc. Amer. Math. Soc. \textbf{142} (2014), no.~2, 575--580,
\href{https://doi.org/10.1090/S0002-9939-2013-11791-0}{doi:10.1090/S0002-9939-2013-11791-0}.
\bibitem {Haagerup1981}U.~Haagerup, \emph{The best constants in the Khintchine
inequality}, Studia Math. \textbf{70} (1981), no.~3, 231--283,
\href{https://doi.org/10.4064/sm-70-3-231-283}{doi:10.4064/sm-70-3-231-283}.
\bibitem {MarshallOlkinArnold2011}A.~W. Marshall, I.~Olkin and B.~C. Arnold,
\emph{Inequalities: Theory of Majorization and Its Applications}, 2nd ed.,
Springer Series in Statistics, Springer, New York, 2011,
\href{https://doi.org/10.1007/978-0-387-68276-1}{doi:10.1007/978-0-387-68276-1}.
\bibitem {Montanaro2012}A.~Montanaro,
\emph{Some applications of hypercontractive inequalities in quantum information theory},
J. Math. Phys. \textbf{53} (2012), no.~12, 122206, 15~pp.,
\href{https://doi.org/10.1063/1.4769269}{doi:10.1063/1.4769269}.
\bibitem {PellegrinoJNT2016}D.~Pellegrino, \emph{The optimal constants of the
mixed $(\ell_{1},\ell_{2})$-Littlewood inequality}, J. Number Theory
\textbf{160} (2016), 11--18,
\href{https://doi.org/10.1016/j.jnt.2015.08.007}{doi:10.1016/j.jnt.2015.08.007}.
\bibitem {PellegrinoTeixeira2018}D.~Pellegrino and E.~V. Teixeira,
\emph{Towards sharp Bohnenblust--Hille constants}, Commun. Contemp. Math.
\textbf{20} (2018), no.~3, 1750029, 33~pp.,
\href{https://doi.org/10.1142/S0219199717500298}{doi:10.1142/S0219199717500298}.
\bibitem {Rado1952}\textcolor{black}{R.~Rado, \emph{An inequality},
J. London Math. Soc. \textbf{27} (1952), no.~1, 1--6,
\href{https://doi.org/10.1112/jlms/s1-27.1.1}{doi:10.1112/jlms/s1-27.1.1}.}

\bibitem {SloteVolbergZhang2024}J.~Slote, A.~Volberg and H.~Zhang,
\emph{Bohnenblust--Hille inequality for cyclic groups}, Adv. Math. \textbf{452}
(2024), Paper No.~109824,
\href{https://doi.org/10.1016/j.aim.2024.109824}{doi:10.1016/j.aim.2024.109824}.

\end{thebibliography}
\end{document}